\documentclass[a4paper,11pt]{amsart}
\usepackage{amssymb}
\usepackage{amscd}
\usepackage{comment}
\usepackage{amsmath,amsthm}
\usepackage[colorlinks=true]{hyperref}
\usepackage{enumerate}
\usepackage{booktabs,multirow}
\usepackage{tikz}
\usepackage{rotating}
\usetikzlibrary{patterns}
\usetikzlibrary{decorations.pathreplacing}
\usetikzlibrary{calc,through}
\usetikzlibrary{trees,arrows,matrix,fit}
\allowdisplaybreaks[1]%allows to break equations if really necessary
\numberwithin{figure}{section}
\numberwithin{equation}{section}

\title{On the lattice of weighted partitions}
\author[K.~Shigechi]{Keiichi~Shigechi}
\email{k1.shigechi AT gmail.com}
\date{\today}
\newcommand\tikzpic[2]{
\raisebox{#1\totalheight}{
\begin{tikzpicture}
#2
\end{tikzpicture}
}}
\newtheorem{theorem}[figure]{Theorem}%%[section]
\newtheorem{example}[figure]{Example}
\newtheorem{lemma}[figure]{Lemma}
\newtheorem{defn}[figure]{Definition}
\newtheorem{prop}[figure]{Proposition}
\newtheorem{cor}[figure]{Corollary}

\newtheorem{remark}[figure]{Remark}
\begin{document}
\begin{abstract}
We introduce and study the lattice of generalized partitions, called weighted partitions.
This lattice possesses similar properties of the lattice of partitions.
By use of the pictorial representation of a weighted partition, the total number 
is given by the successive Stirling transforms of the Stirling number of the second kind. 
We construct an explicit $EL$-labeling on the lattice, which implies this lattice is $EL$-shellable and hence shellable.
We compute the M\"obius function and the characteristic polynomial by use of a pictorial representation of a 
maximal decreasing chain. 
Further, a maximal decreasing chain is shown to be bijective to a labeled rooted complete binary tree.
\end{abstract}

\maketitle

\section{Introduction}
The lattice $\Pi_{n}$ of partitions of the set $[n]:=\{1,2,\ldots,n\}$
is a well-studied combinatorial object.
The numbers of elements of rank $r$ in $\Pi_{n}$ are known as 
the Stirling numbers of the second kind, and the total number is 
the Bell numbers. 
The M\"obius function $\mu(\Pi_{n})$ of $\Pi_{n}$ and 
the characteristic polynomial $\chi_{\Pi_{n}}(x)$ are 
given by \cite{Sta97b1}:
\begin{align*}
\mu(\Pi_{n})&=(-1)^{n-1}(n-1)!, \\ 
\chi_{\Pi_{n}}(x)&=\prod_{j=1}^{n-1}(x-j).
\end{align*}
One way to compute the polynomial $\chi_{\Pi_{n}}(x)$ and $\mu(\Pi_{n})$ is 
to construct an explicit edge-labeling, introduce by Stanley \cite{Sta74}, in the partially ordered set (poset), which 
is called $EL$-labeling.
Then, by a general theory of an edge labeling \cite{Bjo80,Bjo82,Bjo92,BjoGarSta82}, to compute $\mu(\Pi_{n})$ is 
equivalent to count the number of chains with a certain property.
The existence of an $EL$-labeling implies that the partially ordered set is $EL$-shellable, shellable, and 
hence Cohen--Macaulay.  
Besides these, the topological aspect of $\Pi_{n}$, the homology of $\Pi_{n}$, 
has been studied. 
For example, it was shown in \cite{Bjo80} that the only non-vanishing homology groups 
are in top dimension.
Since the poset $\Pi_{n}$ admits a natural action of the symmetric group $\mathbb{S}_{n}$, 
the homology groups of top dimension can be regarded as an $\mathbb{S}_{n}$-module.
By the computation of the representations of $\mathbb{S}_{n}$, the homology groups 
are $\mathbb{S}_{n}$-modules isomorphic to the tensor product of the multi-linear component of the free
Lie algebra of $n$ generators and the sign representation of $\mathbb{S}_{n}$ (see e.g. \cite{Bar90,BarBer90} 
and references therein).

There are several generalizations of $\Pi_{n}$. 
A weighted version of $\Pi_{n}$ is introduced in Dotsenko and Khoroshkin
in the study of binary operads \cite{DotKho07}, and further studied by  Gonz\'alez D'Le\'on and Wachs 
\cite{GonDLeWac16}.
In \cite{Bar90,BarBer90,BjoWac83,GonDLeWac16,Wac98}, it is shown that several properties of the weighted partitions, which 
are similar to those of the partitions: 
$EL$-shellability \cite{BjoWac83}, the M\"obius function by use of rooted trees, bases for homology and 
cohomology \cite{Wac98}, and connection to the free Lie algebra with two types of brackets \cite{GonDLeWac16} 
by generalizing correspondence between the partition lattice and free Lie algebra \cite{Bar90,BarBer90}. 
In \cite{Shi22}, the author introduced the non-commutative crossing partitions, which 
are another generalization of partitions. In the theory of $\Pi_{n}$, the order of the blocks 
of a partition is irrelevant. However, in the non-commutative crossing partitions, 
the order of the blocks plays an important role. By generalizing the notion of 
the parking function, the M\"obius function is calculated. 

In this paper, we introduce and study another generalization of a partition, 
which we call a weighted partition by abuse of notation.
Suppose that a partition has $r$ blocks $B_1,\ldots, B_{r}$, where $B_{i}$ is a 
non-empty set of $[n]$, $B_{i}\cap B_{j}=\emptyset$ for $i\neq j$
and $\bigcup_{i}B_{i}=[n]$.
Similarly, given $k\ge2$, a weighted partition has $kr$ blocks $B_{1}^{(k')},\ldots,B_{r}^{(k')}$
with $1\le k'\le k$, where the blocks $B_{i}^{(1)}$ are non-empty, and some blocks for $k'\ge2$ may be empty
or contain at least two elements.
If we take $k$ blocks $B_{i}^{(k')}$, $1\le k'\le k$, for some $i\in[r]$, 
the block $B_{i}^{(k'+1)}$ is again a partition of the block $B_{i}^{(k')}$.
Each block $B_{i}^{(k')}$ is called a $k'$-th layer for a given $i$.
We denote by $\mathcal{P}_{n}^{(k)}$ be the set of weighted partitions 
which have $k$ layers (see Section \ref{sec:wp} for a detailed explanation). 

Let $s(n,r)$ and $S(n,r)$ be the Stirling numbers of the first and the second kinds.
Recall that the number of partitions of $[n]$ with $r$ blocks is given by the number $S(n,r)$, 
and the number of permutations of $n$ elements with $r$ disjoint cycles is given by $(-1)^{n-r}s(n,r)$.
In the case of the weighted partitions in $\mathcal{P}_{n}^{(k)}$, the number $T(n,k,r)$ of weighted partitions with $r$ blocks
is given by $k$ successive Stirling transforms of $S(n,r)$.
Thus, a weighted partition is bijective to a tuple of partitions.
We also introduce a combinatorial object, a $k$-level labeled rooted trees with $n$ leaves, 
and show that a labeled tree is bijective to a weighted partition in $\mathcal{P}_{n}^{(k)}$ up to an equivalence relation.
Since $S(n,r)$ and $s(n,r)$ are the inverses of each other, the inverse $t(n,k,r)$ of $T(n,k,r)$ is given by 
$k$ successive inverse Stirling transforms of $s(n,r)$.
We also introduce a diagram for the numbers $t(n,k,r)$, and show that the diagram is bijective to 
a tuple of disjoint cycles.
As a consequence of the Stirling transforms, we give formulae for exponential generating functions of 
$t(n,k,r)$ and $T(n,k,r)$. These expressions lead to the recurrence relations for 
the numbers $t(n,k,r)$ and $T(n,k,r)$.

By extending the definitions of the covering relation and the rank function of $\Pi_{n}$ 
to $\mathcal{P}_{n}^{(k)}$, we show that the partially ordered set 
$\mathcal{P}_{n}^{(k)}$ gives a graded lattice $\mathcal{L}_{n}^{(k)}$  by adding the maximum element
to $\mathcal{P}_{n}^{(k)}$.
We first give an explicit $EL$-labeling of the lattice $\mathcal{L}_{n}^{(k)}$, which is a generalization 
of the $EL$-labeling for $\Pi_{n}$ in \cite{Bjo80,Sta74}.
The existence of an $EL$-labeling implies that the lattice $\mathcal{L}_{n}^{(k)}$ is shellable, 
and the M\"obius function is given by the number of maximal decreasing chains (see e.g. \cite{Bjo80}) in the 
Hasse diagram of $\mathcal{L}_{n}^{(k)}$.
By introducing a diagram which is similar to the one for $t(n,k,r)$, we give formulae for the M\"obius 
function $\mu(\mathcal{L}_{n}^{(k)})$ of $\mathcal{L}_{n}^{(k)} 
$and the characteristic polynomials $P_{n}^{(k)}(x)$ 
of $\mathcal{P}_{n}^{(k)}$. Namely, by Theorem \ref{thrm:Gamma} and Proposition \ref{prop:charpol},
we have 
\begin{align*}
\mu(\mathcal{L}_{n}^{(k)})&=(-1)^{n}\prod_{j=0}^{n-2}(k(j+1)-1), \\
P_{n}^{(k)}(x)&=\prod_{j=0}^{n-1}(x-kj).
\end{align*}
It may be interesting to observe that characteristic polynomials involve the Stirling 
numbers of the first kind, which comes from the similarity of the diagrams.
We also show that a maximal decreasing chain of $\mathcal{L}_{n}^{(k)}$ 
is bijective to a labeled complete binary trees with $n$ leaves.
A labeled complete binary tree appearing in this paper is a generalization 
of that in \cite{Bar90}.

The paper is organized as follows.
In Section \ref{sec:wp}, we define a weighted partition and study 
enumerations by introducing pictorial representations.
We show that the number of weighted partitions of rank $r$ is given by
the Stirling transforms of the Stirling number $S(n,r)$ of the 
second kind.
In Section \ref{sec:cr}, we introduce a covering relation for 
weighted partitions, and show that the poset $\mathcal{L}_{n}^{(k)}$ 
is a graded lattice.
In Section \ref{sec:EL}, we give an explicit $EL$-labeling for the lattice, 
and summarize its consequence.
Finally, in Section \ref{sec:CM}, we compute the M\"obius function and 
the characteristic polynomial for $\mathcal{L}_{n}^{(k)}$.
We also give a bijective correspondence between a maximal decreasing chain 
and a labeled complete binary tree.

\section{Weighted partitions}
\label{sec:wp}
\subsection{Definition}
A partition $\pi$ of the set $[n]:=\{1,2,\ldots,n\}$ consists of 
$m$ non-empty blocks $\{B_{p}: p\in[m]$\} such that 
the blocks satisfy $B_{p}\cap B_{q}=\emptyset$ if $p\neq q$ and 
$\bigcup_{p\in[m]} B_{p}=[n]$.
Since each integer $i\in[n]$ appears exactly once in $B_{p}$, 
we write $B_{p}$ as $i_{1}\ldots i_{l}$ where $l$ is the cardinality 
of $B_{p}$ and $i_{j}<i_{j+1}$ for $j=1,\ldots,l-1$.
A partition $\pi$ is written as a concatenation of $B_{p}$'s where 
we separate $B_{p}$ and $B_{p+1}$ by a slash ``/".
We call this one-line notation of $\pi$.
For example, $\pi=135/24/6$ means that we have three blocks 
$B_{1}=\{1,3,5\}$, $B_{2}=\{2,4\}$ and $B_{3}=\{6\}$.
Further, we ignore the order of the blocks. 
This means that $135/24/6$ is equivalent to $6/24/135$ and other several expressions.

Suppose that a partition $\pi$ of the set $[n]$ consists of $p$ blocks.
Since each block can be regarded as a set of integers, one can 
consider a partition of the elements of the block.
We call a partition of the block a block of the second layer.
The second layer of a partition also can be regarded as a set of integers.
We can consider a partition of the block in the second layer and 
obtain a block of the third layer.
Then, by repeating the above procedure, one can obtain a block of
the $k$-th layer.
A block in the first layer contains at least one elements.
We consider the following condition on a partition:
\begin{enumerate}
\item[($\star$)]
A block in the $k'$-th layer with $2\le k'\le k$ is an empty set or contains sets consisting of 
at least two elements.
\end{enumerate}

\begin{defn}
\label{defn:Pnk}
We denote $\pi^{(k)}$ as a partition consisting of up to the $k$-th layer and 
satisfying the condition ($\star$), and call it \emph{a weighted partition}.
We denote by $\mathcal{P}^{(k)}_{n}$ the set of weighted partitions $\pi^{(k)}$ of
the set $[n]$.
Similarly, we define $\mathcal{P}_{n}^{(k)}(r)$ as the set of weighted partitions 
with $r$ blocks.
\end{defn}
Note that we have standard partitions of $n$ if we set $k=1$.

We denote by $B^{(k)}_{p}$ a block of a partition of the $k$-th layer.
The one-line notation of $\pi^{(k)}$ is similarly defined as in the 
case of a partition.
For example, $\pi^{(3)}=1(35)^{2}/(24)^3/6$ 
means that 
\begin{align*}
&B^{(1)}_1=\{1,3,5\}, \quad  B^{(1)}_2=\{2,4\}, \quad B^{(1)}_3=\{6\}, \\
&B^{(2)}_1=\{3,5\}, \quad B^{(2)}_2=\{2,4\}, \quad B^{(2)}_3=\emptyset\\
&B^{(3)}_1=\emptyset, \quad B^{(3)}_2=\{2,4\} \quad B^{(3)}_{3}=\emptyset.
\end{align*}
Note that we have $B_{i}^{(1)}\supseteq B_{i}^{(2)}\supseteq B_{i}^{(3)}$ 
for all $i\in[3]$.

\begin{remark}
Two remarks are in order:
\begin{enumerate}
\item
By the construction of a weighted partition of the $k$-th layer, 
the block $B_{i}^{(k)}$ consists of blocks of the set $B_{i}^{(k-1)}$.
For example, $(12)^{2}(34)^2$ means that 
$B_{1}^{(1)}=\{1,2,3,4\}$ and $B_{1}^{(2)}=12/34=\{\{1,2\},\{3,4\}\}$.
The block $B_1^{(2)}$ is an refinement of the block $B_{1}^{(1)}$ and 
satisfies the condition ($\star$).
\item 
The condition ($\star$) implies that we have no weighted partition 
such as $13/(2)^24$.
In this example, we have 
\begin{align*}
&B_{1}^{(1)}=\{1,3\}, \quad B_{2}^{(1)}=\{2,4\}, \\
&B_{1}^{(2)}=\emptyset, \quad B_{2}^{(2)}=\{2\}. 
\end{align*}
The block $B_{2}^{(2)}$ in the second layer violates the condition ($\star$).
\end{enumerate}
\end{remark}

For later purpose, we introduce a circular pictorial notation of 
weighted partitions.
Recall that a partition $\pi^{(k)}$ consists of $mk$ blocks $B^{(j)}_{p}$ where 
$p\in[m]$ and $j\in[k]$.
Let $C$ be a circle with $n$ points which are labeled $1$ to $n$ clockwise.	
Since $B^{(k)}_{p}$ consists of several blocks $B'_{q}$, we connect 
all pairs of labels, which appear in $B'_{q}$, by lines and put an integer 
$k$ on the lines.
This is well-defined since the block $B'_q$ contains at least two elements 
by the condition ($\star$).
Since $B^{(k)}_{p}$ is a refinement of $B^{(k-1)}_{p}$, there may be 
pairs of labels in $B^{(k-1)}_{p}$ which do not appear as a pair in $B^{(k)}_{p}$.
We connect such pairs in $C$ by a line, and put an integer $k-1$.
By repeating this procedure by decreasing $k$ one-by-one, we obtain 
a circular presentation of a weighted partition $\pi^{(k)}$.
The circular presentation of $\pi^{(3)}=1(35)^{2}/(24)^{3}/6$ 
is depicted in Figure \ref{fig:cr}.
The pair of labels $(2,4)$ is in the third layer (connected by a red line), $(3,5)$ is 
in the second layer (connected by a blue line), and the pairs $(1,3)$ and $(1,5)$ are 
in the first layer.

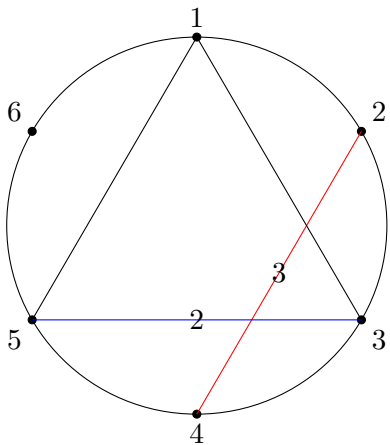
\begin{figure}[ht]
\begin{tikzpicture}
\draw circle(2.5cm);
\foreach \a in {90,30,-30,-90,-150,-210}
\filldraw [black] (\a:2.5cm)circle(1.5pt);
\draw (90:2.5cm) node[anchor=south] {$1$};
\draw (30:2.5cm) node[anchor=south west] {$2$};
\draw (-30:2.5cm) node[anchor=north west] {$3$};
\draw (-90:2.5cm) node[anchor=north] {$4$};
\draw (-150:2.5cm) node[anchor=north east] {$5$};
\draw (-210:2.5cm) node[anchor=south east] {$6$};
\draw[red](30:2.5cm)to(-90:2.5cm);
\draw ($(30:2.5cm)!0.5!(-90:2.5cm)$)node{$3$};
\draw (-150:2.5cm) to (90:2.5cm) to (-30:2.5cm);
\draw[blue] (-150:2.5cm) to (-30:2.5cm);
\draw ($(-150:2.5cm)!0.5!(-30:2.5cm)$) node{$2$};
\end{tikzpicture}
\caption{A circular presentation of $\pi^{(3)}=1(35)^{2}/(24)^{3}/6$. 
The blue (resp. red) line indicates the second (resp. third) layer.}
\label{fig:cr}
\end{figure}

\subsection{Enumeration}
We briefly review Stirling numbers of the first and second kind, and 
their properties.

The falling factorial $(x)_{n}:=x(x-1)\cdots(x-n+1)$ is expanded in terms 
of the Stirling numbers $s(n,k)$ of the first kind as 
\begin{align}
\label{eq:defsn}
(x)_{n}=\sum_{k=0}^{n}s(n,k)x^{k}.
\end{align}
Equivalently, we have 
\begin{align*}
\sum_{n\ge k}s(n,k)\genfrac{}{}{}{}{x^n}{n!}=\genfrac{}{}{}{}{1}{k!}(\log(1+x))^{k}.
\end{align*}
The number $|s(n,k)|$ counts the number of permutations in $[n]$ with $k$ disjoint 
cycles.
The exponential generating function of $s(n,k)$ is given by
\begin{align}
\label{eq:egfS1}
\sum_{n,k}s(n,k)\genfrac{}{}{}{}{x^n}{n!}y^{k}=(1+x)^{y}.
\end{align}

The Stirling numbers $S(n,k)$ of the second kind is defined 
by 
\begin{align*}
S(n,k):=\genfrac{}{}{}{}{1}{k!}\sum_{i=0}^{k}(-1)^{i}\genfrac{(}{)}{0pt}{}{k}{i}(k-i)^{n}.
\end{align*}
The numbers $S(n,k)$ count the number to partition the set $[n]$ into $k$ non-empty 
subsets.
The exponential generating function of $S(n,k)$ with a fixed $k$ is given by
\begin{align*}
\sum_{n\ge k}S(n,k)\genfrac{}{}{}{}{x^n}{n!}=\genfrac{}{}{}{}{1}{k!}(e^{x}-1)^{k}.
\end{align*}
Similarly, the  exponential generating function for $S(n,k)$ is given by
\begin{align}
\label{eq:ogfS2}
\sum_{n,k}S(n,k)\genfrac{}{}{}{}{x^n}{n!}y^{k}
=\exp(y(e^{x}-1)).
\end{align}

The Stirling numbers of the first and second kinds are inverse of each other, namely, 
we have 
\begin{align}
\label{eq:sS}
\sum_{j\ge0}s(n,j)S(j,k)=\delta_{nk},
\end{align}
and
\begin{align}
\label{eq:Ss}
\sum_{j\ge0}S(n,j)s(j,k)=\delta_{nk},
\end{align}
where $\delta_{nk}$ is the Kronecker delta function.

We define the set $I(n,k,r)$ by 
\begin{align*}
I(n,k,r):=\left\{\mathbf{i}:=(i_1,\ldots,i_{k-1})| n\ge i_1\ge i_2\ge\ldots\ge i_{k-1}\ge r \right\}.
\end{align*}

We define the numbers $T(n,k,r)$ by 
\begin{align*}
T(n,k,r):=\sum_{\mathbf{i}\in I(n,k,r)}\prod_{j=1}^{k}S(i_{j-1},i_j),
\end{align*}
where $i_0=n$ and $i_{k}=r$.
Similarly, we define the numbers $t(n,k,r)$ by 
\begin{align*}
t(n,k,r):=\sum_{\mathbf{i}\in I(n,k,r)}\prod_{j=1}^{k}s(i_{j-1},i_j),
\end{align*}
where $i_0=n$ and $i_{k}=r$.

As in the case of the Stirling numbers of the first and second kinds,
the two numbers $T(n,k,r)$ and $t(n,k,r)$ are inverse of one another:
\begin{align*}
\sum_{i\ge0}t(n,k,i)T(i,k,j)=\delta_{nj},
\end{align*}
and
\begin{align*}
\sum_{i\ge0}T(n,k,i)t(i,k,j)=\delta_{nj}.
\end{align*}
These identities can be easily shown by use of relations (\ref{eq:sS}) and (\ref{eq:Ss}).

To express the generating functions of $T(n,k,r)$ and $t(n,k,r)$ in a simple form,
we introduce two formal power series $\exp^{(k)}(x,y)$ and $\log^{(k)}(x,y)$.
Let $\exp^{(k)}(x)$, $k\ge1$,	 be a formal power series 
given by 
\begin{align*}
\exp^{(k)}(x)=\exp(\exp^{(k-1)}(x))-1,
\end{align*}
with the initial condition $\exp^{(0)}(x)=\exp(x)-1$.
Similarly, we define a multivariate formal power series $\exp^{(k)}(x,y)$ 
by 
\begin{align}
\label{eq:expk}
\exp^{(k)}(x,y):=\exp(y \exp^{(k-1)}(x)).
\end{align} 

Similarly, let $\log^{(k)}(x)$, $k\ge1$, be the formal power series
given by 
\begin{align*}
\log^{(k)}(x):=1+\log(\log^{(k-1)}(x)),
\end{align*}
with the initial condition $\log^{(0)}(x):=\exp(x)$.
Then, we define 
\begin{align*}
\log^{(k)}(x,y):=(\log^{(k)}(x))^{y}.
\end{align*}

Let us briefly recall the Stirling transforms.
Suppose that two sequences of numbers $\{a_n: 1\le n\}$ and $\{b_{n}: 1 \le n\}$
satisfy the relation:
\begin{align*}
b_{n}=\sum_{k=1}^{n}S(n,k)a_{k}.
\end{align*}
Then, since $S(n,k)$ and $s(n,k)$ are inverse of each other, the inverse Stirling transform 
is given by 
\begin{align*}
a_{n}=\sum_{k=1}^{n}s(n,k)b_{k}.
\end{align*}
If $A(x):=\sum_{n}a_{n}x^n/n!$ and $B(x):=\sum_{n}b_{n}x^n/n!$,
the Stirling transforms and generating functions of Stirling numbers imply that 
\begin{align}
\label{eq:STS1S2}
B(x)=A(e^x-1), \qquad A(x)=B(\log(1+x)).
\end{align}
We apply the Stirling transforms to the Stirling numbers of the first and second kinds 
to obtain the numbers $t(n,k,r)$ and $T(n,k,r)$.
%%%%%%%%%%%%%%%%%%%%%%%%%%%%%%%55
\begin{lemma}
\label{lemma:Tnkr}
The generating function of the numbers $T(n,k,r)$ is given by
\begin{align}
\label{eq:egfTnkr}
\sum_{n,r}T(n,k,r)\genfrac{}{}{}{}{x^{n}}{n!}y^{r}=\exp^{(k)}(x,y).
\end{align}
Similarly, the generating function of the numbers $t(n,k,r)$ is 
given by 
\begin{align}
\label{eq:egftnkr}
\sum_{n,r}t(n,k,r)\genfrac{}{}{}{}{x^{n}}{n!}y^{r}=\log^{(k)}(x,y).
\end{align}
\end{lemma}
%%%%%%%%%
\begin{proof}
We first show Eq. (\ref{eq:egfTnkr}) by induction on $k$.
For $k=1$, $T(n,k=1,r)$ is the Stirling number of the second kind,
and Eq. (\ref{eq:egfTnkr}) follows from Eq. (\ref{eq:ogfS2}).
By definition of $T(n,k,r)$, the sequence $T(n,k,r)$ is the Stirling 
transform of $T(n,k-1,r)$.
From the Stirling transform (\ref{eq:STS1S2}), we have Eq. (\ref{eq:egfTnkr}).

Equation (\ref{eq:egftnkr}) can be shown by use of the inverse Stirling transform.
\end{proof}

First few terms of $\exp^{(k)}(x,y)$ and $\log^{(k)}(x,y)$ are as follows:
\begin{align*}
\exp^{(1)}(x,y)&=1+yx+(y+y^2)\genfrac{}{}{}{}{x^2}{2!}+(y+3y^2+y^3)\genfrac{}{}{}{}{x^3}{3!}
+(y+7y^2+6y^3+y^4)\genfrac{}{}{}{}{x^4}{4!}+\cdots, \\
\exp^{(2)}(x,y)&=1+yx+(2y+y^2)\genfrac{}{}{}{}{x^2}{2!}+(5y+6y^2+y^3)\genfrac{}{}{}{}{x^3}{3!}
+(15y+32y^2+12y^3+y^4)\genfrac{}{}{}{}{x^4}{4!}+\cdots, \\
\exp^{(3)}(x,y)&=1+yx+(3y+y^2)\genfrac{}{}{}{}{x^2}{2!}+(12y+9y^2+y^3)\genfrac{}{}{}{}{x^3}{3!}
+(60y+75y^2+18y^3+y^4)\genfrac{}{}{}{}{x^4}{4!}+\cdots, 
\end{align*}
and 
\begin{align*}
\log^{(1)}(x,y)&=1+yx+(-y+y^2)\genfrac{}{}{}{}{x^2}{2!}+(2y-3y^2+y^3)\genfrac{}{}{}{}{x^3}{3!}
+(-6y+11y^2-6y^3+y^4)\genfrac{}{}{}{}{x^4}{4!}+\cdots,\\
\log^{(2)}(x,y)&=1+yx+(-2y+y^2)\genfrac{}{}{}{}{x^2}{2!}+(7y-6y^2+y^3)\genfrac{}{}{}{}{x^3}{3!}
+(-35y+40y^2-12y^3+y^4)\genfrac{}{}{}{}{x^4}{4!}+\cdots,\\
\log^{(3)}(x,y)&=1+yx+(-3y+y^2)\genfrac{}{}{}{}{x^2}{2!}+(15y-9y^2+y^3)\genfrac{}{}{}{}{x^3}{3!} \\
&\quad+(-105y+87y^2-18y^3+y^4)\genfrac{}{}{}{}{x^4}{4!}+\cdots.
\end{align*}

\begin{example}
Let $(n,k)=(3,2)$.
We have one weighted partition with three block: $3/2/1$.
We have six weighted partitions with two blocks: $23/1$, $2/13$, $3/12$, $(23)^2/1$, 
$2/(13)^2$, and $3/(12)^2$.
Similarly, we have five weighted partitions with one block:
$123$, $1(23)^2$, $2(13)^2$, $3(12)^2$ and $(123)^2$. 
From these, we have $T(3,2,3)=1$, $T(3,2,2)=6$ and $T(3,2,1)=5$.
In total, we have $12$ weighted partitions for $(n,k)=(3,2)$.
\end{example}

\begin{remark}
The specialization $\exp^{(1)}(x,y=1)$ gives the exponential generating 
function of the {\it Bell numbers}.
The Bell numbers $B_n$ are given by the sum of the Stirling numbers $S(n,k)$ of the second kind:
\begin{align*}
B_{n}=\sum_{k=0}^{n}S(n,k),
\end{align*}
and satisfy a recurrence relation 
\begin{align*}
B_{n+1}=\sum_{k=0}^{n}\genfrac{(}{)}{0pt}{}{n}{k}B_{k},
\end{align*}
with the initial condition $B_{0}=1$.
\end{remark}

The next proposition gives the number of weighted partitions in $\mathcal{P}_{n}^{(k)}(r)$. 
\begin{prop}
\label{prop:numwp}
Let $1\le r$ be the number of blocks in the first layer in a weighted partition $\pi^{(k)}$.
Then, the number of such weighted partitions is given by $T(n,k,r)$, namely, 
\begin{align*}
|\mathcal{P}_{n}^{(k)}(r)|=T(n,k,r).
\end{align*}
Especially, the number of weighted partitions $T(n,k,r=1)$ is 
given by $\sum_{r=1}^{n}T(n,k-1,r)$, namely, 
\begin{align*}
|\mathcal{P}_{n}^{(k)}(r=1)|=\sum_{r=1}^{n}T(n,k-1,r).
\end{align*}
\end{prop}
%%%%%%%%%%
\begin{proof}
Recall that the Stirling number $S(n,n_1)$ of the second kind counts the number
to partition $[n]$ into $n_1$ non-empty sets.
To specify a weighted partition $\pi$, we first partition $[n]$ into non-empty sets
$B_{i}^{(1)}$, $1\le i\le n_1$, which corresponds to the first layer of $\pi$.
We label the edges in the circular presentation by $1$. 

We give a recursive construction of a weighted partition $\pi^{(k)}\in\mathcal{P}_{n}^{(k)}$
from a pair of weighted partitions $(\pi^{(k-1)},\pi')$ where $\pi^{(k-1)}\in\mathcal{P}_{n}^{(k-1)}$ and 
$\pi'\in\mathcal{P}_{m}^{(1)}$ where $m$ is the number of blocks in the first layer of $\pi^{(k-1)}$.
Suppose that we have a circular presentation of $\pi^{(k-1)}\in\mathcal{P}_{n}^{(k-1)}$ where 
$\pi^{(k-1)}$ has $n_1$ blocks in the first layer. 
We first increase the labels of $\pi^{(k-1)}$ by one, and denote its circular presentation by $\nu^{(2,k)}$.
By definition of $T(n,k,r)$, we have $T(n,k,r)=\sum_{n_1}T(n,k-1,n_1)S(n_1,r)$.
Let $\pi'\in\mathcal{P}_{n_1}^{(1)}$ be a partition of $[n_1]$ with $r$ blocks.
The total number of partitions $\pi'$ corresponds to $S(n_1,r)$.
Let $B_{i}^{(1)}$, $1\le i\le n_1$, be the blocks in the first layer of $\pi^{(k-1)}$ and 
suppose they satisfy that $\min(B_{i}^{(1)})<\min(B_{i+1}^{(1)})$ for all $1\le i\le n_1-1$.
Recall $\pi'$ is a partition of $n_1$ with $r$ blocks.
If integers $i_1,i_2,\ldots, i_{p}$ form a block in $\pi'$, we define 
a block $B':=B^{(1)}_{i_1}\cup B^{(1)}_{i_2}\cup\ldots\cup B^{(1)}_{i_{p}}$.
Note that the set of blocks $B'$ gives a partition of $[n]$. 
In this way, one can construct a partition $\nu^{(1)}$ of $[n]$ with $r$ blocks $B'$.
Then, we define a weighted partition $\pi^{(k)}$ such that 
its first layer is $\nu^{(1)}$ and from the second to the $k$-th layers are $\nu^{(2,k)}$.
This construction implies that we have a bijection between a partition $\pi^{(k)}$ and 
a pair $(\pi^{(k-1)},\pi')$.
Therefore, the number of weighted partitions in $\mathcal{P}_{n}^{(k)}$ is 
given by the sum of the products of the Stirling number of the second kind,
namely, $T(n,k,r)$.

The number of weighted partitions $\pi^{(k)}$ with $r=1$ in $\mathcal{P}_{n}^{(k)}$ 
is equal to the number of weighted partitions $\pi^{(k-1)}$ in $\mathcal{P}_{n}^{(k-1)}$
since if we ignore the first layer of the partitions $\pi^{(k)}$, we have a natural 
bijection between $\pi^{(k)}$ and $\pi^{(k-1)}$.
Then, it follows that the total number $T(n,k,r=1)$ of such $\pi^{(k)}$ is equal to 
$\sum_{r=1}^{n}T(n,k-1,r)$.
\end{proof}

\begin{example}
\label{ex:Tdiagram}
An example of construction of a weighted partition $\pi^{(k)}\in\mathcal{P}_{n}^{(k)}$ 
from a pair of weighted partitions $(\pi^{(k-1)},\pi')\in\mathcal{P}_{n}^{(k)}\times \mathcal{P}_{m}^{(1)}$.
Let $\pi^{(2)}\in\mathcal{P}_{6}^{(2)}$ and $\pi' \in\mathcal{P}_{3}^{(1)}$ 
be the following weighted partitions:
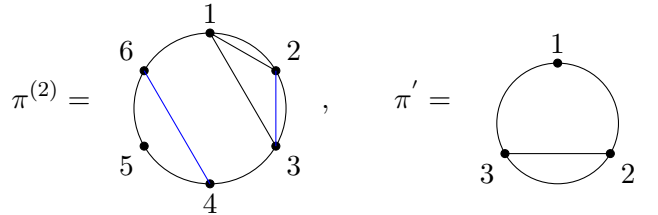
\begin{align*}
\pi^{(2)}=
\tikzpic{-0.5}{
\draw circle(1cm);
\foreach \a in {90,30,-30,-90,-150,150}
\filldraw[black](\a:1cm)circle(1.5pt);
\draw (90:1cm) node[anchor=south] {$1$};
\draw (30:1cm) node[anchor=south west]{$2$};
\draw (-30:1cm)node[anchor=north west]{$3$};
\draw (-90:1cm)node[anchor=north]{$4$};
\draw (-150:1cm)node[anchor=north east]{$5$};
\draw (150:1cm)node[anchor=south east]{$6$};
\draw(90:1cm)--(30:1cm)(90:1cm)--(-30:1cm);
\draw[blue](30:1cm)--(-30:1cm)(-90:1cm)--(150:1cm);
}, \qquad
\pi^{'}
=
\tikzpic{-0.5}{
\draw circle(0.8cm);
\foreach \a in {90,-30,-150}
\filldraw[black](\a:0.8cm)circle(1.5pt);
\draw (90:0.8cm) node[anchor=south] {$1$};
\draw (-30:0.8cm) node[anchor=north west] {$2$};
\draw (-150:0.8cm)node[anchor=north east] {$3$};
\draw (-30:0.8cm)--(-150:0.8cm);
},
\end{align*} 
where blue lines correspond to the second layer.
The one line notation of $\pi^{(2)}$ is $1(23)^2/(46)^{2}/5$, and 
the minimum elements in the three blocks are $1,4$ and $5$.
From $\pi'$, we merge the blocks containing $4$ and $5$ into a block, and 
obtain a block $(46)^35$.
Thus, a weighted partition in $\mathcal{P}_{6}^{(3)}$ corresponding to 
the pair $(\pi^{(2)},\pi')$ is given by 
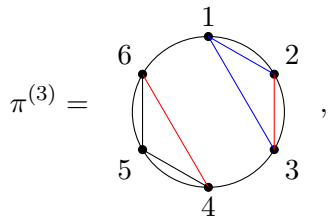
\begin{align*}
\pi^{(3)}=
\tikzpic{-0.5}{
\draw circle(1cm);
\foreach \a in {90,30,-30,-90,-150,150}
\filldraw[black](\a:1cm)circle(1.5pt);
\draw (90:1cm) node[anchor=south] {$1$};
\draw (30:1cm) node[anchor=south west]{$2$};
\draw (-30:1cm)node[anchor=north west]{$3$};
\draw (-90:1cm)node[anchor=north]{$4$};
\draw (-150:1cm)node[anchor=north east]{$5$};
\draw (150:1cm)node[anchor=south east]{$6$};
\draw[blue](90:1cm)--(30:1cm)(90:1cm)--(-30:1cm);
\draw[red](30:1cm)--(-30:1cm)(-90:1cm)--(150:1cm);
\draw (-90:1cm)--(-150:1cm)--(150:1cm);	
},
\end{align*}
where a red (resp. blue) line corresponds to the third (resp. second) layer.
The one-line notation is $\pi^{(3)}=1^2(23)^3/(46)^35$.
\end{example}

\subsection{Recurrence relations}
To obtain a recurrence relations for the numbers $T(n,k,r)$ and $t(n,k,r)$,
we introduce some notions regarding to a partition.

Let $\lambda:=(\lambda_1,\lambda_2,\ldots,\lambda_{l})$ be 
a partition of $n$, that is, $\sum_{i=1}^{l}\lambda_{i}=n$, 
$\lambda_1\ge\lambda_2\ge\ldots\ge\lambda_{l}>0$, and 
$l:=l(\lambda)$ is the length of $\lambda$.
We abbreviate $\lambda\vdash n$ if $\lambda$ is a partition of $n$.
Given a partition $\lambda\vdash n$, we define 
\begin{align*}
m_{i}:=\#\{j: \lambda_{j}=i, 1\le j\le l(\lambda)\},
\end{align*}
and we write $\lambda=(1^{m_1},2^{m_{2}},\ldots,n^{m_{n}})$.
Then, we define the two values $f_{\lambda}$ and $g_{\lambda}$ by
\begin{align*}
f_{\lambda}&:=\genfrac{}{}{}{}{n!}{\prod_{j=1}^{n}(j!)^{m_{j}} m_{j}!}, \\
g_{\lambda}&:=(l(\lambda)-1)!\cdot f_{\lambda},
\end{align*}
where $l(\lambda)$ is the length of $\lambda$.
For example, in the case of $\lambda=(2,2,1)\vdash 5$, 
we have $l(\lambda)=3$, $f_{(2,2,1)}=5!/((2!)^2\cdot 2!)=15$, and $g_{(2,2,1)}=30$.

For general $k\ge 0$, we define the elementary symmetric function 
\begin{align*}
e_{k}(x_1,\ldots,x_n):=
\sum_{1\le j_1<j_2<\ldots<j_{k}\le n}x_{j_1}\cdots x_{j_k}.
\end{align*}

We introduce diagrams which are associated to the numbers $t(n,k,r)$.

Recall that the Stirling numbers $s(n,k)$ of the first kind count
the number of permutations with $k$ disjoint cycles.
Let $S_1$ be a circle with $n$ points labeled $1,2,\ldots,n$ clock-wise.
The number $s(n,r)=t(n,k=1,r)$ counts the following combinatorial objects $C^{(1)}$ on $S_{1}$:
\begin{enumerate}
\item $C$ consists of $n-r$ directed lines from the point $i$ to the point $j$ such that $i<j$.
The direction of the line is from $i$ to $j$.
\item At each point $i\in[n]$, there is at most one incoming lines.
\item Every directed line has an integer label $1$.
\end{enumerate}
We say that the points $i_1,i_2,\ldots,i_{m}$ are connected  
if any two points $i_{p}$ and $i_q$ are connected by several lines.
We define the size of a connected component by $m$.
Further, we also say the points $i_1,\ldots, i_{m}$ are in the same block 
if they are connected.
The size of a block is given by $m$.
For example, consider the two diagrams for $(n,r)=(4,2)$: 
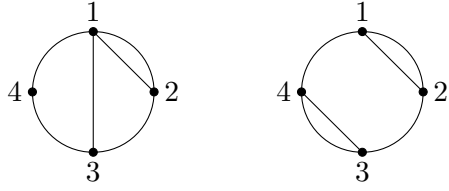
\begin{align*}
\tikzpic{-0.5}{
\draw circle(0.8cm);
\foreach \a in {90,0,-90,-180}
\filldraw[black](\a:0.8cm)circle(1.5pt);
\draw (90:0.8cm) node[anchor=south] {$1$};
\draw (0:0.8cm) node[anchor=west] {$2$};
\draw (-90:0.8cm)node[anchor=north] {$3$};
\draw (-180:0.8cm)node[anchor=east]{$4$};
\draw (90:0.8cm)--(-90:0.8cm)(90:0.8cm)--(0:0.8cm);
}\qquad 
\tikzpic{-0.5}{
\draw circle(0.8cm);
\foreach \a in {90,0,-90,-180}
\filldraw[black](\a:0.8cm)circle(1.5pt);
\draw (90:0.8cm) node[anchor=south] {$1$};
\draw (0:0.8cm) node[anchor=west] {$2$};
\draw (-90:0.8cm)node[anchor=north] {$3$};
\draw (-180:0.8cm)node[anchor=east]{$4$};
\draw (90:0.8cm)--(0:0.8cm)(-90:0.8cm)--(-180:0.8cm);
}
\end{align*}
The both diagrams consist of two blocks.
In the left diagram, we have one block with size three and a block of size one consisting of 
a single point. In the right diagram, we have two blocks of size two.

\begin{defn}
\label{defn:Cnr}
We denote by $\mathcal{C}(n,r)$ the set of diagrams $C^{(1)}$ with $n$ points and 
$r$ blocks, which satisfy the conditions from (1) to (3). 
\end{defn}
By definition, the cardinality of $\mathcal{C}(n,r)$ is given by the 
(signless) Stirling number $|s(n,r)|$ of the first kind.
Each connected component of $C\in\mathcal{C}(n,r)$ corresponds to 
a single disjoint circle.

Diagrams $C^{(k)}$ on $S_1$ for $t(n,k,r)$ are recursively constructed from 
a pair $(C^{(k-1)},C^{(1)})$ where the diagrams $C^{(k-1)}$ on $S_1$ for $t(n,k-1,r')$ 
where $r\le r'\le n$ and $C^{(1)}\in\mathcal{C}(r',r)$.
From the definition $t(n,k,r)=\sum_{r'}t(n,k-1,r')s(r',r)$, we first increase the integer labels on 
directed lines in $C^{(k-1)}$ by one.
Note that $C^{(k-1)}$ has $r'$ blocks.
Then, we add $r'-r$ directed lines with an integer label $1$ as follows:
\begin{enumerate}[($\heartsuit1$)]
\item Fix a diagram $C_1$ in $\mathcal{C}(r',r)$.
\item Since the diagram $C^{(k-1)}$ has $r'$ blocks, we denote the blocks by 
$B_{i}$, $1\le i\le r'$.  Let $b_{i}$ be the minimum integer label in $B_{i}$.
\item If the points $i$ and $j$ are connected by a directed line in $C_{1}$, 
we connect $b_{i}$ and $b_{j}$ in $C^{(k-1)}$ by a directed line.
The newly added line has an integer label $1$. 
\end{enumerate}
As a consequence, a diagram $C^{(k)}$ is constructed from 
the pair of diagrams $(C^{(k-1)},C^{(1)})$.
Here, if the number of blocks in $C^{(k-1)}$ is $r$, then 
$C^{(1)}$ has $r$ points in its diagram.

\begin{example}
\label{ex:tdiagram}
Let $C^{(2)}\in\mathcal{C}(6,3)$ and $C_{1}\in\mathcal{C}(3,2)$ be the following diagrams:
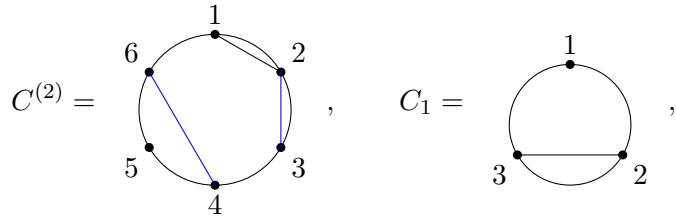
\begin{align*}
C^{(2)}=
\tikzpic{-0.5}{
\draw circle(1cm);
\foreach \a in {90,30,-30,-90,-150,150}
\filldraw[black](\a:1cm)circle(1.5pt);
\draw (90:1cm) node[anchor=south] {$1$};
\draw (30:1cm) node[anchor=south west]{$2$};
\draw (-30:1cm)node[anchor=north west]{$3$};
\draw (-90:1cm)node[anchor=north]{$4$};
\draw (-150:1cm)node[anchor=north east]{$5$};
\draw (150:1cm)node[anchor=south east]{$6$};
\draw(90:1cm)--(30:1cm);
\draw[blue](30:1cm)--(-30:1cm)(-90:1cm)--(150:1cm);
}, \qquad
C_{1}
=
\tikzpic{-0.5}{
\draw circle(0.8cm);
\foreach \a in {90,-30,-150}
\filldraw[black](\a:0.8cm)circle(1.5pt);
\draw (90:0.8cm) node[anchor=south] {$1$};
\draw (-30:0.8cm) node[anchor=north west] {$2$};
\draw (-150:0.8cm)node[anchor=north east] {$3$};
\draw (-30:0.8cm)--(-150:0.8cm);
},
\end{align*} 
where black lines have an integer label $1$ and blue lines have an integer label $2$.
The minimum labels in the blocks of $C^{(2)}$ are $1,4$ and $5$.
Since the points $2$ and $3$ are connected by a line in $C_{1}$, we connect the points $4$ and $5$
in $C^{(3)}$ by a line.
Then, we have a following diagram $C^{(3)}\in\mathcal{C}(6,2)$:
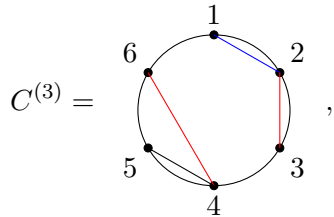
\begin{align*}
C^{(3)}=
\tikzpic{-0.5}{
\draw circle(1cm);
\foreach \a in {90,30,-30,-90,-150,150}
\filldraw[black](\a:1cm)circle(1.5pt);
\draw (90:1cm) node[anchor=south] {$1$};
\draw (30:1cm) node[anchor=south west]{$2$};
\draw (-30:1cm)node[anchor=north west]{$3$};
\draw (-90:1cm)node[anchor=north]{$4$};
\draw (-150:1cm)node[anchor=north east]{$5$};
\draw (150:1cm)node[anchor=south east]{$6$};
\draw[blue](90:1cm)--(30:1cm);
\draw[red](30:1cm)--(-30:1cm)(-90:1cm)--(150:1cm);
\draw(-90:1cm)--(-150:1cm);
},
\end{align*}
where a red line has an integer label $3$.
\end{example}

Let $C\in\mathcal{C}(n,r)$ be a diagram with $n-r$ lines.
Since all $n-r$ lines are directed edges, we have a sequence of points
\begin{align*}
i_1\rightarrow i_2\rightarrow\ldots \rightarrow i_{m},
\end{align*}
where $i_j\rightarrow i_{j+1}$ means that we have a directed edge 
from the point $i_j$ to the point $i_{j+1}$.
We may have several sequences of points starting from the point $i_j$
for some $j$.
Let $c\in[k]$. We put an integer label on the edge $e$ connecting 	
$i_j$ and $i_{j+1}$. Let $c(i_j,i_{j+1})$ be the label on the edge $e$.
%%%%%%%%%%%%%%%%
\begin{lemma}
\label{lemma:condcij}
The label $c(i_j,i_{j+1})$ satisfies the following conditions:
\begin{enumerate}
\item $1\le c(i_j,i_{j+1})\le k$. 
\item $c(i_j,i_{j+1})\le c(i_{j+1},i_{j+2})$ for all $j$.
\end{enumerate}
\end{lemma}
%%%%%%%%%%%%%%%%
\begin{proof}
By condition ($\heartsuit3$), the label of a line in the diagram 
is in $\{1,2,\ldots,k\}$, which implies the condition (1).

Proof of (2). Suppose that we have three points $p_1,p_2$ and $p_3$ in $C^{(k)}$ 
such that $p_1<p_2<p_3$ and a line connecting $p_1$ with $p_2$ 
has a label $l_1$ and a line connecting $p_2$ with $p_3$ has 
a label $l_2$.
We will show that $l_1\le l_2$.
We assume that $l_1>l_2$. 
From the assumption and the condition ($\heartsuit3$), 
there exists $k'<k$ such that the two points $p_1$ and $p_2$ are connected 
by a line in $C^{(k')}$ and two points $p_2$ and $p_3$ are not connected by a line.
In $C^{(k')}$, the minimum label of the block containing two points $p_1$ and $p_2$
is less than or equal to $p_1$.
Thus, from the condition ($\heartsuit2$), we cannot connect the points $p_2$ and 
$p_3$ by a line in $C^{(k)}$. However, this is a contradiction since $p_2$ and $p_3$ are connected 
in $C^{(k)}$.
Thus, we have $l_1\le l_2$.
\end{proof}

\begin{defn}
We call a labeling of the diagram $C$ satisfying the above two 
conditions in Lemma \ref{lemma:condcij} a $k$-labeling.
We denote the number of $k$-labels of a diagram $C$ 
by $\mathrm{wt}^{(k)}(C)$.
\end{defn}

For example, consider the two diagrams in $\mathcal{C}(4,1)$:
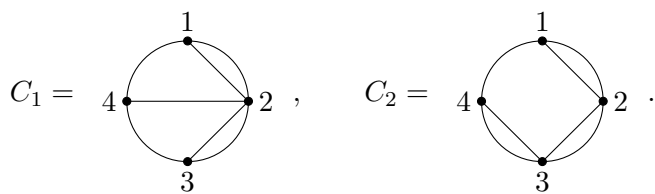
\begin{align*}
C_{1}=
\tikzpic{-0.5}{
\draw circle(0.8cm);
\foreach \a in {90,0,-90,-180}
\filldraw[black](\a:0.8cm)circle(1.5pt);
\draw (90:0.8cm) node[anchor=south] {$1$};
\draw (0:0.8cm) node[anchor=west] {$2$};
\draw (-90:0.8cm)node[anchor=north] {$3$};
\draw (-180:0.8cm)node[anchor=east]{$4$};
\draw (90:0.8cm)--(0:0.8cm)(0:0.8cm)--(180:0.8cm)(0:0.8cm)--(-90:0.8cm);
},\qquad 
C_{2}=\tikzpic{-0.5}{
\draw circle(0.8cm);
\foreach \a in {90,0,-90,-180}
\filldraw[black](\a:0.8cm)circle(1.5pt);
\draw (90:0.8cm) node[anchor=south] {$1$};
\draw (0:0.8cm) node[anchor=west] {$2$};
\draw (-90:0.8cm)node[anchor=north] {$3$};
\draw (-180:0.8cm)node[anchor=east]{$4$};
\draw (90:0.8cm)--(0:0.8cm)--(-90:0.8cm)--(-180:0.8cm);
}.
\end{align*}
We have two sequences of points in $C_{1}$:
\begin{align*}
1\rightarrow 2 \rightarrow 3, \qquad 
1\rightarrow 2 \rightarrow 4.
\end{align*}
Similarly, we have a unique sequence of points in $C_{2}$:
\begin{align*}
1\rightarrow 2\rightarrow 3\rightarrow 4.
\end{align*}
When $k=3$, we have $\mathrm{wt}^{(3)}(C_1)=14$ and 
$\mathrm{wt}^{(3)}(C_{2})=10$.
The value $\mathrm{wt}^{(3)}(C_1)=14$ comes from that 
we have nine $3$-labelings with $c(1,2)=1$, 
four labelings with $c(1,2)=2$, and one labeling with $c(1,2)=3$.
Similarly, the $3$-labels of the diagram $C_2$ are  
six labels with $c(1,2)=1$, three labels with $c(1,2)=2$ and 
one label with $c(1,2)=3$.

The following proposition shows that the numbers $t(n,k,r)$
are the generating function of diagrams $C\in\mathcal{C}(n,r)$ 
with a weight $\mathrm{wt}(C)$.
%%%%%%%%%%%%%%%%
\begin{prop}
The numbers $t(n,k,r)$ satisfy
\begin{align}
\label{eq:tnkrC}
t(n,k,r)=
(-1)^{n+r}\sum_{C\in\mathcal{C}(n,r)}\mathrm{wt}^{(k)}(C),
\end{align}
\end{prop}
%%%%%%%%%%%%
\begin{proof}
If we forget the labels of lines in any diagram $C^{(k)}$, 
we have a diagram $C$ in $\mathcal{C}(n,r)$.
Therefore, to obtain the number $t(n,k,r)$, it is enough 
to count the diagrams $C^{(k)}$ which have the same connectivity 
as $C$.
Lemma \ref{lemma:condcij} implies that the number of such $C^{(k)}$
is the number of $k$-labeling of $C$.
Then, it is easy to see that $t(n,k,r)$ satisfies Eq. (\ref{eq:tnkrC}),
which completes the proof.
\end{proof}

\begin{example}
We have $t(3,3,1)=15$. The diagram for $t(3,3,1)$ consists of a single block 
with $3$ elements.
We have two diagrams whose sequences of points are given by 
\begin{align*}
3\leftarrow 1\rightarrow 2, \qquad 1\rightarrow 2\rightarrow 3.
\end{align*}
In the left sequence, we have $9$ $3$-labelings.
We have $6$ $3$-labelings for the right sequence. In total, we have $15$ $3$-labelings 
with a single block. 
\end{example}

\begin{remark}
We have two circular presentations (see Example \ref{ex:Tdiagram} and 
Example \ref{ex:tdiagram}) for $T(n,k,r)$ and $t(n,k,r)$ respectively.
The difference between two presentations is that we may have a loop 
consisting of several edges for the former, and no loops for the latter.
Further, in the latter case, there exists a unique sequence of edges 
from the integer $j>i$ to $i$ where $i$ is the minimum element of a connected 
component containing $i$ and $j$.
\end{remark}

Below, we give two more recurrence relations for the numbers 
$t(n,k,r)$.

\begin{prop}
\label{prop:rec1tnkr}
The numbers $t(n,k,r)$ satisfies the recurrence relation:
\begin{align}
\label{eq:rec2tnkr}
t(n,k,r)=\sum_{p=0}^{n-r}\genfrac{(}{)}{0pt}{}{n-1}{p}\cdot
t(p+1,k,1)\cdot t(n-p-1,k,r-1),
\end{align}
and 
\begin{align}
\label{eq:tnk1}
t(n,k,1)=\sum_{\lambda\vdash n}(-1)^{l(\lambda)+1}g_{\lambda}
\prod_{i=1}^{l(\lambda)}t(\lambda_i,k-1,1),
\end{align}
with the initial conditions $t(0,k,0)=t(1,k,1)=1$.
\end{prop}
%%%%%%%%%%%
\begin{proof}
We first prove Eq. (\ref{eq:tnk1}).
Since $t(n,1,1)=s(n,1)$ and $s(n,1)$ counts the number of a cycle 
of size $n$ with a sign, we have $t(n,1,1)=(-1)^{n-1}(n-1)!$.
Thus, the exponential generating function of $s(n,1)$ is given 
by $1+\log(1+x)$.
By definition, we have $t(n,k,1)=\sum_{n'}s(n,n')t(n',k-1,1)$.
By applying the Stirling transform to $t(n,k,1)$, it is easy to see 
that 
\begin{align}
\label{eq:rectnk1}
\begin{split}
1+\sum_{n\ge1}t(n,k,1)\genfrac{}{}{}{}{x^n}{n!}&=\log^{(k)}(x), \\
&=1+\log\left(1+\sum_{n\ge 1}t(n,k-1,1)\genfrac{}{}{}{}{x^{n}}{n!}\right).
\end{split}
\end{align}
By expanding the right hand side of Eq. (\ref{eq:rectnk1}) with respect to $x$,
we obtain Eq. (\ref{eq:tnk1}).

Proof of Eq. (\ref{eq:rec2tnkr}).
By construction of diagrams associated to the numbers $t(n,k,r)$, a diagram $C^{(k)}$
consists of several blocks. 
Let $B_{1}$ be the block containing the point $1$ in $C^{(k)}$.
Suppose that $B_{1}$ has $p$ directed lines. It implies that we have $p+1$ points in 
this block.
Since we focus on $B_1$, 
we have $\genfrac{(}{)}{0pt}{}{n-1}{p}$ ways to choose $p$ points from $n-1$ points.
Then, we have $t(p+1,k,1)$ ways to form a cycle of size $p+1$.
Besides $B_{1}$, we have $n-p-1$ points and these points should form $r-1$ disjoint 
cycles, which implies we have the factor $t(n-p-1,k,r-1)$. 
From these observations, we have Eq. (\ref{eq:rec2tnkr}).
\end{proof}

The next proposition gives the recurrence relation for $t(n,k,r)$ in terms of elementary symmetric 
functions.
\begin{prop}
We have
\begin{align*}
t(n,k,r)=
\sum_{r\le a}(-1)^{a-r}e_{a-r}(1,2,\ldots,a-1)
\sum_{\genfrac{}{}{0pt}{}{\lambda\vdash n}{l(\lambda)=a}}
f_{\lambda}\prod_{j=1}^{l(\lambda)}t(\lambda_{j},k-1,1),
\end{align*}
where $e_{a-r}(1,2,\ldots,a-1)$ is a specialization of 
$e_{a-r}$ at $x_{i}=i$.
\end{prop}
%%%%%%%%%%%%%
\begin{proof}
By substituting 
\begin{align*}
\log^{(k)}(x)=1+\sum_{1\le n}t(n,k-1,1)\genfrac{}{}{}{}{x^n}{n!},
\end{align*}
into Eq. (\ref{eq:egftnkr}), the recurrence relation follows.
\end{proof}
%%%%%%%%
\begin{remark}
Note that the specialization of the elementary symmetric function 
is equal to a Stirling number of the first kind:
\begin{align*}
e_{a-r}(1,2\ldots,a-1)=(-1)^{a-r}s(a,r).
\end{align*}
\end{remark}

We give a recurrence relation for $T(n,k,r)$.
By substituting 
\begin{align*}
\exp^{(k-1)}(x)=\sum_{1\le n}\genfrac{}{}{}{}{x^{n}}{n!}\sum_{r'=1}^{n}T(n,k-1,r'),
\end{align*}
into the definition (\ref{eq:expk}) of $\exp^{(k)}(x,y)$, 
it is easy to see that the expression (\ref{eq:egfTnkr}) is equivalent to 
the following expression:
\begin{prop}
The numbers $T(n,k,r)$ satisfies the recurrence relation
\begin{align}
\label{eq:recTnkr}
T(n,k,r)
=\sum_{\genfrac{}{}{0pt}{}{\lambda\vdash n}{l(\lambda)=r}}
f_{\lambda}\prod_{i=1}^{r}\left(
\sum_{r'=1}^{\lambda_{i}}T(\lambda_{i},k-1,r')
\right),
\end{align}
with the initial conditions $T(0,k,0)=T(1,k,1)=0$.
\end{prop}
We interpret Eq. (\ref{eq:recTnkr}) in terms of weighted partitions.
Recall that $T(n,k,r)$ counts the number of weighted partitions such that 
it has $k$ layers and $r$ blocks in the first layer.
Below, we will show the recurrence relation (\ref{eq:recTnkr}) holds 
for the number of weighted partitions.
Suppose $\lambda\vdash n$ and $r:=l(\lambda)$. 
By definition, the numbers $f_{\lambda}$ count the weighted partitions $\pi$
such that $\pi$ has $r$ blocks, and each block has $\lambda_{i}$ elements 
where $1\le i\le r$.
We regard these $r$ blocks as the first layer $\nu^{(1)}$ of $\pi$.
To obtain a weighted partition $\pi$ from $\nu^{(1)}$, 
we have to overlay $k-1$ layers on $\nu^{(1)}$.
Since each block in $\nu^{(1)}$ has $\lambda_{i}$ elements, 
we has $\sum_{r'}T(\lambda_{i},k-1,r')$ ways to put layers. 
This is because the numbers $T(\lambda_{i},k-1,r')$ count 
weighted partitions of $[\lambda_{i}]$ which have $k-1$ layers with $r'$ blocks.  
From these observations, it is easy to see that Eq. (\ref{eq:recTnkr})
holds for weighted partitions.

We have an analogue of Proposition \ref{prop:rec1tnkr} for $T(n,k,r)$.
\begin{prop}
The numbers $T(n,k,r)$ satisfy the recurrence relations:
\begin{align}
\label{eq:rec2Tnkr}
T(n,k,r)=\sum_{p=0}^{n-r}\genfrac{(}{)}{0pt}{}{n-1}{p}T(p+1,k,1)T(n-p-1,k,r-1),
\end{align} 
and 
\begin{align}
\label{eq:rec2Tnk1}
T(n,k,1)=\sum_{r=1}^{n}T(n,k-1,r),
\end{align}
with initial conditions $T(0,k,0)=T(1,k,1)=1$.
\end{prop}
%%%%%%%%%%%%%%%%%%%
\begin{proof}
From Proposition \ref{prop:numwp}, it is easy to see that 
the numbers $T(n,k,1)$ satisfy Eq. (\ref{eq:rec2Tnk1}). 

Equation (\ref{eq:rec2Tnkr}) can be proven by the same 
argument as the proof of Proposition \ref{prop:rec1tnkr}.
\end{proof}

\subsection{Weighted partitions and rooted trees}
Fix $(n,k)$ with $n\ge2$ and $k\ge2$.
A {\it $k$-level rooted tree} $T_n$ with $n$ leaves is a rooted planar tree 
with $n$ leaves such that there are exactly $k+1$ edges from a leaf to the root.
We label a leaf of $T(n)$ by an integer $i\in[n]$.
An integer in $[n]$ appears exactly once in leaves.
We call $T_n$ with labels {\it $k$-level labeled rooted trees} with $n$ leaves.
\begin{defn}
We denote by $\mathcal{T}^{(k)}_{n}$ the set of $k$-level rooted 
trees with $n$ leaves, and by $\mathcal{L'}_{n}^{(k)}$ the set of 
$k$-level labeled rooted trees with $n$ leaves.
\end{defn}

Let $T\in\mathcal{T}_{n}^{(k)}$, and 
suppose that a node $c$ has $p$ children and they are the root of sub-trees 
$T_1,\ldots, T_{p}$ in $T$.
Let $S_1,\ldots, S_{q}$ be the set of trees such that
\begin{enumerate}
\item $S_{r}\neq S_{r'}$ if $r\neq r'$.
\item Given $i\in[p]$, there exists $j\in[q]$ such that $T_{i}=S_{j}$.
\item There is at least one $i\in[p]$ such that $S_{j}=T_{i}$.
\end{enumerate}
By definition, $\{S_{j}: j\in[q] \}$ is the set of trees which appear 
in $\{T_{i}: i\in[p] \}$. We always have $p\ge q$.

Let $m_{1},\ldots,m_{q}$ be the number 
\begin{align}
\label{eq:defmj}
m_{j}:=\#\{T_{i} | T_{i}=S_{j}, 1\le i\le p \}.
\end{align}
By definition, we have $m_{j}\ge1$ for $j\in[q]$ and $\sum_{j=1}^{q}m_{j}=p$.
We define a {\it symmetric factor} $F(c)$ of the node $c$ by
\begin{align*}
F(c):=\prod_{j=1}^{q}(m_{j}!)^{-1}.
\end{align*}
The factor $F(c)$ comes from the symmetry of a $k$-level rooted tree at the node $c$.
If $T_a=T_b$ for $a,b\in[p]$, the tree is invariant under the exchange of $T_a$ and $T_{b}$. 

Similarly, we introduce an equivalence relation on $L(T)\in\mathcal{L'}_{n}^{(k)}$.
Since the tree structure of the labeled tree $L(T)$ is $T$, we adopt the same notation 
$T_{i}$ and $S_j$ as before.
We denote by $L(T:i\leftrightarrow i')$ the labeled tree obtained from $L(T)$ by 
exchanging $T_{i}$ and $T_{i'}$. 
Then, we write the equivalence relation as 
\begin{align}
\label{eq:equivrel}
L(T)\sim L(T,i\leftrightarrow i').
\end{align}

\begin{defn}
We define 
\begin{align*}
\mathcal{LRT}_{n}^{(k)}:=\mathcal{L'}_{n}^{(k)}/\sim.
\end{align*}
\end{defn}

\begin{prop}
The set $\mathcal{LRT}_{n}^{(k)}$ is bijective to the set $\mathcal{P}_{n}^{(k)}$ of weighted partitions.
\end{prop}
%%%%%%%%%%
\begin{proof}
We construct a weighted partition $\pi\in\mathcal{P}_{n}^{(k)}$ from 
$L(T)\in\mathcal{LRT}_{n}^{(k)}$, and show that the map is a bijection.

Let $c$ be the node at depth $k'$, i.e., there are $k'$ edges from $c$ to the root.
Let $T(c)$ be the sub-tree whose root is $c$.
Then, the labels on the leaves in $T(c)$ belong to the same block in the $k'$-th 
layer in $\pi$.
Then, it is obvious that $\pi$ is a weighted partition in $\mathcal{P}_{n}^{(k)}$.
Since we ignore the order of blocks in $\pi$, the equivalence relation (\ref{eq:equivrel})
follows.

Conversely, given a weighted partition $\pi\in\mathcal{P}_{n}^{(k)}$, one can construct a $k$-level 
labeled rooted tree from $\pi$ by inverting the above map.
\end{proof}

\begin{example}
For $(n,k)=(3,2)$, we have $12$ weighted partitions.
We have six $2$-level rooted trees with $3$ leaves in $\mathcal{LRT}^{(2)}_{3}$:
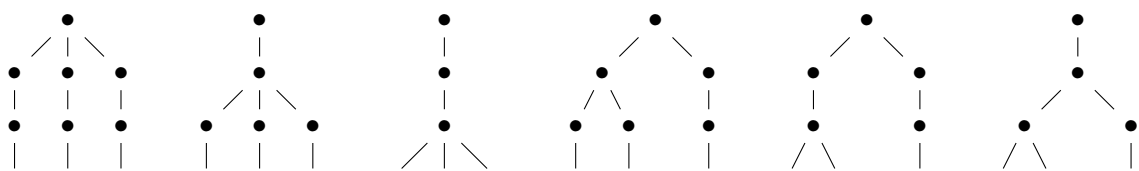
\begin{align*}
\tikzpic{-0.5}{[scale=0.7]
\node (root) at(0,0){$\bullet$};
\node (l1)at(-1,-1){$\bullet$};
\node (l2)at(-1,-2){$\bullet$};
\node(l3)at(-1,-3){};
\node (c1)at(0,-1){$\bullet$};
\node(c2)at(0,-2){$\bullet$};
\node(c3)at(0,-3){};
\node (r1)at(1,-1){$\bullet$};
\node (r2)at(1,-2){$\bullet$};
\node(r3)at(1,-3){};
\draw(root)--(l1)--(l2)--(l3)(root)--(c1)--(c2)--(c3)(root)--(r1)--(r2)--(r3);
}
\quad
\tikzpic{-0.5}{[scale=0.7]
\node (root) at(0,0){$\bullet$};
\node (root2)at(0,-1){$\bullet$};
\node (l2)at(-1,-2){$\bullet$};
\node(l3)at(-1,-3){};
\node(c2)at(0,-2){$\bullet$};
\node(c3)at(0,-3){};
\node (r2)at(1,-2){$\bullet$};
\node(r3)at(1,-3){};
\draw(root)--(root2)--(l2)--(l3)(root2)--(c2)--(c3)(root2)--(r2)--(r3);
}
\quad
\tikzpic{-0.5}{[scale=0.7]
\node (root) at(0,0){$\bullet$};
\node (root2)at(0,-1){$\bullet$};
\node (root3)at(0,-2){$\bullet$};
\node(l3)at(-1,-3){};
\node(c3)at(0,-3){};
\node(r3)at(1,-3){};
\draw(root)--(root2)--(root3)--(l3)(root3)--(c3)(root3)--(r3);
}
\quad
\tikzpic{-0.5}{[scale=0.7]
\node(root)at(0,0){$\bullet$};
\node(l1)at(-1,-1){$\bullet$};
\node(l2)at(-1.5,-2){$\bullet$};
\node(l3)at(-1.5,-3){};
\node(c1)at(-0.5,-2){$\bullet$};
\node(c2)at(-0.5,-3){};
\node(r1)at(1,-1){$\bullet$};
\node(r2)at(1,-2){$\bullet$};
\node(r3)at(1,-3){};
\draw(root)--(l1)--(l2)--(l3)(l1)--(c1)--(c2)(root)--(r1)--(r2)--(r3);
}
\quad
\tikzpic{-0.5}{[scale=0.7]
\node(root)at(0,0){$\bullet$};
\node(l1)at(-1,-1){$\bullet$};
\node(root2)at(-1,-2){$\bullet$};
\node(l3)at(-1.5,-3){};
\node(c2)at(-0.5,-3){};
\node(r1)at(1,-1){$\bullet$};
\node(r2)at(1,-2){$\bullet$};
\node(r3)at(1,-3){};
\draw(root)--(l1)--(root2)--(l3)(root2)--(c2)(root)--(r1)--(r2)--(r3);
}
\quad
\tikzpic{-0.5}{[scale=0.7]
\node(root)at(0,0){$\bullet$};
\node(root2)at(0,-1){$\bullet$};
\node(root3)at(-1,-2){$\bullet$};
\node(l3)at(-1.5,-3){};
\node(c2)at(-0.5,-3){};
\node(r2)at(1,-2){$\bullet$};
\node(r3)at(1,-3){};
\draw(root)--(root2)--(root3)--(l3)(root3)--(c2)
(root2)--(r2)--(r3);
}
\end{align*}
The number of labelings in each tree is $1,1,1,3,3$ and $3$ from 
left to right. In total, we have $12$ representatives for $2$-level labeled rooted trees.
\end{example}

\begin{example}
\label{ex:LRT}
We consider the same weighted partition in Figure \ref{fig:cr}.
The corresponding $3$-level labeled rooted tree is 
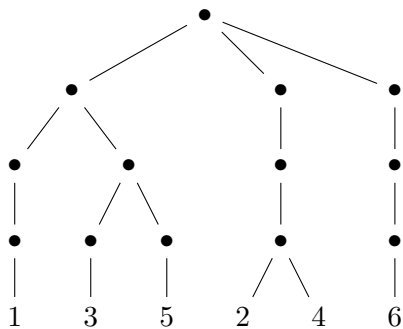
\begin{align*}
\tikzpic{-0.5}{
\node (root) at (0,0){$\bullet$};
\node (l1) at (-1.75,-1){$\bullet$};
\node (c1) at (1,-1){$\bullet$};
\node (r1) at (2.5,-1){$\bullet$};
\node (ll2)at (-2.5,-2){$\bullet$};
\node (l2) at (-1,-2){$\bullet$};
\node (r2) at (1,-2) {$\bullet$};
\node (rr2) at (2.5,-2){$\bullet$};
\node (lll3)at (-2.5,-3){$\bullet$};
\node (ll3) at (-1.5,-3){$\bullet$};
\node (l3)at (-0.5,-3){$\bullet$};
\node (r3)at(1,-3){$\bullet$};
\node (rr3)at (2.5,-3){$\bullet$};
\node(lll4) at (-2.5,-4){$1$};
\node(ll4) at (-1.5,-4){$3$};
\node(l4) at (-0.5,-4){$5$};
\node(r4) at (0.5,-4){$2$};
\node(rr4) at (1.5,-4){$4$};
\node(rrr4) at (2.5,-4){$6$};
\draw(root)--(l1)--(ll2)--(lll3)--(lll4)(l1)--(l2)--(ll3)--(ll4)(l2)--(l3)--(l4);
\draw(root)--(c1)--(r2)--(r3)--(r4)(r3)--(rr4);
\draw(root)--(r1)--(rr2)--(rr3)--(rrr4);
}
\end{align*} 
The labeled rooted tree is invariant under the exchange of the labels $3$ and $5$, or $2$ and $4$.  
We obtain the weighted partition $1(35)^2/(24)^3/6$ from this $3$-level labeled rooted tree as follows.
Since the integers $2$ and $4$ collide at the third level, we have $(24)^{3}$.
Similarly, the integers $3$ and $5$ collide at the second level, we have $(35)^2$.
Then, since the integers $1$ and $(35)$ collide at the first level, we have $1(35)^2$.
Three blocks $1(35)^2$, $(24)^{3}$ and $6$ collide at the zeroth level, which gives the 
weighted partition $1(35)^2/(24)^3/6$.
\end{example}

\begin{remark}
In a $k$-level labeled rooted tree, we have the root at the zeroth level.
This means that we have a partition $12\ldots n$ at the zeroth layer.
If a node of $k'$-th level with $k'\ge1$ has at least two children, 
we merge the blocks below this node at the $k'$-th layer.
These imply the condition ($\star$) for a weighted partition. 
\end{remark}

\begin{prop}
\label{prop:LRT}
Let $T\in\mathcal{T}_{n}^{(k)}$ and $L(T)\subset\mathcal{LRT}_{n}^{(k)}$ the set of representatives of 
$k$-level labeled rooted trees whose tree structure is $T$.
Then, 
\begin{align}
|L(T)|=n!\prod_{c\in T}F(c).
\end{align}
\end{prop}
%%%%%%%%%%%%%%%
\begin{proof}
Given $T$, we have $n!$ ways to put integer labels on the leaves of $T$.
Since we fix the tree structure, the equivalence relation (\ref{eq:equivrel})
implies that a node $c$ has $\prod_{j}(m_j!)$ equivalent ways to put labels, 
where $m_j$ is defined in Eq. (\ref{eq:defmj}).
Thus, to obtain the number of representatives of labeled trees, 
we have to divide $n!$ by $\prod_{j}(m_j!)$.
This completes the proof.
\end{proof}

\begin{example}
We consider the $3$-level labeled rooted tree as in Example \ref{ex:LRT}.
In the tree, we have two nodes whose symmetric functor is $2!=4$.
Then, by Proposition \ref{prop:LRT}, we have $6!/(2!)^2=180$ representatives.
\end{example}

\section{Covering relation}
\label{sec:cr}
Let $\pi^{(k)}\in\mathcal{P}_{n}^{(k)}$ and suppose that 
$\pi^{(k)}$ has $n-r$ blocks in the first layer.
The integer $r$ takes a value in $\{0,1,\ldots,n-1\}$.

\begin{defn}
\label{defn:rk}
The function $\rho:\mathcal{P}_{n}^{(k)}\rightarrow \mathbb{Z}_{\ge0}$ is given 
by $\rho(\pi^{(k)}):=r$, where a weighted partition $\pi^{(k)}$ has $n-r$ blocks.
\end{defn}
%%%%%%%%%%%%%%%%

\begin{remark}
\label{remark:rho}
To define the function $\rho$, we do not need the information about the 
$k'$-th layers with $2\le k'\le k$.
This property will become clear when we regard the function $\rho$ 
as a rank function. The function $\rho$ is compatible with the definition 
of a cover relation in the case of $k\ge2$.
\end{remark}

Below, we will see that the function $\rho$ is a rank function of $\mathcal{P}_{n}^{(k)}$.

We first recall the covering relation for $\mathcal{P}^{(1)}_{n}$.
Let $\pi_i\in\mathcal{P}_{n}^{(1)}$, $i=1,2$, and 
$\rho(\pi_{2})=\rho(\pi_{1})+1$. 
Suppose that $\pi_1$ has $m$ blocks $B_1,\ldots, B_{m}$ and 
$\pi_2$ has $m-1$ blocks $B'_{1},\ldots, B'_{m-1}$.
We say that a partition $\pi_{1}$ is a refinement of a partition$\pi_2$
if there exists a unique $q$ such that $B'_{q}=B_{i_1}\cup B_{i_2}$ and 
other blocks $B'_{p}$, $p\neq q$, coincide with $B_{j}$ for 
some $j\neq i_1, i_2$.

\begin{defn}
A partition $\pi_2$ covers $\pi_1$ if and only if $\pi_1$ is a refinement of $\pi_2$
and $\rho(\pi_2)=\rho(\pi_1)+1$.
We write $\pi_1\lessdot\pi_2$.
\end{defn}
Then, it is obvious that the function $\rho$ is a rank function on $\mathcal{P}_{n}^{(1)}$.

Given two partitions $\pi_1$ and $\pi_2$, we write 
$\pi_1\le\pi_2$ if there exists a sequence of partitions 
$\pi_1=x_{0}\lessdot x_1\lessdot x_2\lessdot \ldots \lessdot x_{m}=\pi_2$.
It is a routine to show that $P^{(1)}_{n}$ is a partially ordered set 
with the homogeneous relation $\le$.

Further, the poset $(\mathcal{P}_{n}^{(1)},\le)$ is a bounded graded lattice.
The greatest element $\hat{1}$ is $[n]$, and 
the least element $\hat{0}$ is $1/2/\ldots/n$.
Since the function $\rho$ on $\mathcal{P}_{n}^{(1)}$ is a rank 
function, the lattice has a natural grading by $\rho$.

Figure \ref{fig:n3k1} shows the Hasse diagram of the lattice $(\mathcal{P}_{3}^{(1)},\le)$.
The least element is in the left-most column and the greatest element
is in the right-most column.

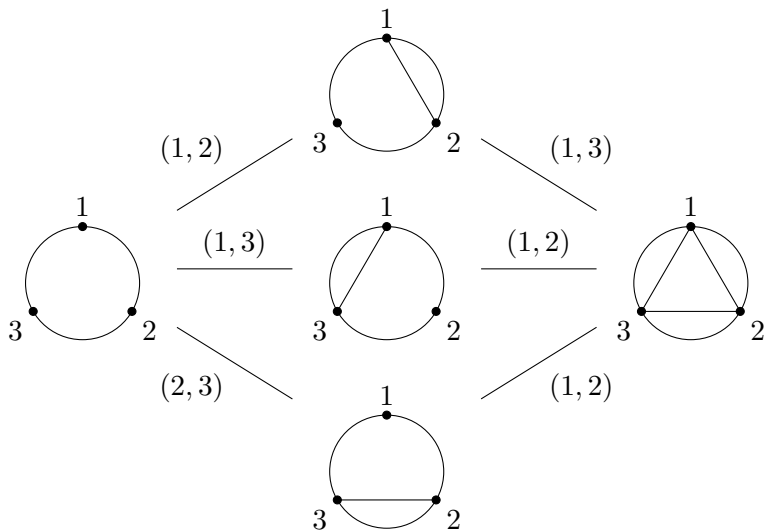
\begin{figure}[ht]
\begin{tikzpicture}
\node (0) at (0,0){
\begin{tikzpicture}\draw circle(0.75cm);
\foreach \a in {90,210,-30}
\filldraw [black] (\a:0.75cm)circle(1.5pt);
\draw(90:0.75cm)node[anchor=south]{$1$};
\draw(-30:0.75cm)node[anchor=north west]{$2$};
\draw(210:0.75cm)node[anchor=north east]{$3$};
\end{tikzpicture}
};
\node (1) at (4,2.5){
\begin{tikzpicture}\draw circle(0.75cm);
\foreach \a in {90,210,-30}
\filldraw [black] (\a:0.75cm)circle(1.5pt);
\draw (90:0.75cm) to (-30:0.75cm);
\draw(90:0.75cm)node[anchor=south]{$1$};
\draw(-30:0.75cm)node[anchor=north west]{$2$};
\draw(210:0.75cm)node[anchor=north east]{$3$};
\end{tikzpicture}
};
\node (2) at (4,0){
\begin{tikzpicture}\draw circle(0.75cm);
\foreach \a in {90,210,-30}
\filldraw [black] (\a:0.75cm)circle(1.5pt);
\draw (90:0.75cm) to (210:0.75cm);
\draw(90:0.75cm)node[anchor=south]{$1$};
\draw(-30:0.75cm)node[anchor=north west]{$2$};
\draw(210:0.75cm)node[anchor=north east]{$3$};
\end{tikzpicture}
};
\node (3) at (4,-2.5){
\begin{tikzpicture}\draw circle(0.75cm);
\foreach \a in {90,210,-30}
\filldraw [black] (\a:0.75cm)circle(1.5pt);
\draw (210:0.75cm) to (-30:0.75cm);
\draw(90:0.75cm)node[anchor=south]{$1$};
\draw(-30:0.75cm)node[anchor=north west]{$2$};
\draw(210:0.75cm)node[anchor=north east]{$3$};
\end{tikzpicture}
};
\node (4) at (8,0){
\begin{tikzpicture}\draw circle(0.75cm);
\foreach \a in {90,210,-30}
\filldraw [black] (\a:0.75cm)circle(1.5pt);
\draw (90:0.75cm) to (210:0.75cm) to (-30:0.75cm) to (90:0.75cm);
\draw(90:0.75cm)node[anchor=south]{$1$};
\draw(-30:0.75cm)node[anchor=north west]{$2$};
\draw(210:0.75cm)node[anchor=north east]{$3$};
\end{tikzpicture}
};
\draw(0) to node[anchor=south east]{$(1,2)$} (1)to node[anchor=south west]{$(1,3)$} (4) 
(0) to node[anchor=south]{$(1,3)$} (2) to node[anchor=south]{$(1,2)$} (4) 
(0) to node[anchor=north east]{$(2,3)$} (3) to node[anchor=north west]{$(1,2)$} (4);
\end{tikzpicture}
\caption{The lattice for $(n,k)=(3,1)$}
\label{fig:n3k1}
\end{figure}

Suppose that $\pi_1\lessdot\pi_2$ and a block $B'_{q}$ of $\pi_2$ is 
expressed as $B'_q=B_{i_1}\cup B_{i_2}$ where $B_{i_1}$ and $B_{i_2}$ are 
the blocks of $\pi_1$.
We define two positive integers $\alpha$ and $\beta$ by 
\begin{align}
\label{eq:alphabeta}
\alpha:=\min(B_{i_1}\cup B_{i_2}), \qquad \beta:=\max(\min(B_{i_1}),\min(B_{i_2})).
\end{align}
To specify a layer of blocks, we write $l$ as a subscript to the right of 
the pair of integers $(\alpha,\beta)$, namely, we write $(\alpha,\beta)_{l}$ 
where $l\in[k]$.
We assign a pair of integers $(\alpha,\beta)_{l}$ with a subscript to the edge connecting 
$\pi_1$ with $\pi_{2}$ in the Hasse diagram.
In Figure \ref{fig:n3k1}, the edges in the Hasse diagram have a pair of integers.
We omit the subscript since the layers of the edges are all the same in this case.

Suppose $\pi_1\lessdot\pi_2$ and $(\alpha,\beta)$ is a pair of integers which is 
assigned to the edge in the Hasse diagram.
We write this covering relation as
\begin{align*}
\pi_1\xrightarrow{(\alpha,\beta)}\pi_2.
\end{align*}
When we have 
\begin{align}
\label{eq:chainab}
\pi_0\xrightarrow{(\alpha_1,\beta_1)}\pi_{1}\xrightarrow{(\alpha_2,\beta_2)}\pi_2
\xrightarrow{(\alpha_3,\beta_3)}\ldots\xrightarrow{(\alpha_m,\beta_m)}\pi_m,
\end{align}
we abbreviate Eq. (\ref{eq:chainab})
\begin{align}
\label{eq:abbchain}
\pi_0:\genfrac{(}{)}{0pt}{}{\alpha_1}{\beta_1}\genfrac{(}{)}{0pt}{}{\alpha_2}{\beta_2}
\ldots\genfrac{(}{)}{0pt}{}{\alpha_m}{\beta_m}.
\end{align}
We sometimes omit $\pi_0$ in Eq. (\ref{eq:abbchain}) if it is obvious from the context 
or we emphasize only labels $(\alpha_{i},\beta_{i})$. 
When we specify the layer of a label, we put a subscript to the right of the label,
namely, we write $\genfrac{(}{)}{0pt}{}{\alpha_i}{\beta_i}_{l}$ where $l$ is the 
label of a layer.
We call Eq. (\ref{eq:abbchain}) a {\it chain from the partition $\pi_0$}, or 
simply a {\it chain}.

In Figure \ref{fig:n3k1}, a maximal chain from $\hat{0}=1/2/3$ to $123$ has two edges 
and note that $\beta$ in a label is either $2$ or $3$ and $2$ and $3$ appear exactly 
once in the chain. This observation is easily generalized as the next lemma.
\begin{lemma}
\label{lemma:betaperm}
Suppose $k=1$. 
In a maximal chain from $\hat{0}$ to $12\ldots n$, $\beta$ in the label $(\alpha,\beta)$ of the edges 
in the Hasse diagram takes a value in $\{2,3,\ldots,n\}$ exactly once.
\end{lemma}
%%%%%%%%%%
\begin{proof}
From the definition (\ref{eq:alphabeta}) of $\alpha$ and $\beta$, 
it is obvious that $\beta$ takes a value in $\{2,3,\ldots,n\}$. 
We show that the integers $\{2,3\ldots,n\}$ appear exactly once in $\beta$ 
by induction of $n$. 
Since we consider a maximal chain, we have 
\begin{align*}
\hat{0}\xrightarrow{(a,b)}\pi_1\xrightarrow{}\ldots\xrightarrow{} 12\ldots n,
\end{align*}
where $\rho(\pi_1)=1$. Since $\hat{0}\xrightarrow{(a,b)}\pi_1$ and 
$\pi_1$ contains a block consisting of $a$ and $b$, $b$ does not 
appear in the labels of the chain from $\pi_1$ to $12\ldots n$.
The partition $\pi_1$ can be regarded as a minimal element in $[n]\setminus\{b\}$ by 
regarding a block consisting of $a$ and $b$ as a block consisting of only $a$.
This does not effect the labels in a maximal chain from $\pi_1$ to $12\ldots n$.
Then, we consider a maximal chain from $\pi_1$ to $12\ldots n$.
By induction assumption, the integers in $[n]\setminus\{1,b\}$ appears exactly once 
in the label $\beta$. Since $b$ appears in the chain from $\hat{0}$ to $\pi_1$,
the integers in $[2,n]$ appear exactly once as a label $\beta$ in the maximal chain.
\end{proof}

The integer $\alpha$ in a label takes a value in $\{1,2,\ldots,n-1\}$, 
and its multiplicity can be greater than one.

\begin{remark}
In Figure \ref{fig:n3k1}, note that we have two commutative labels: 
\begin{align*}
\hat{0}:\genfrac{(}{)}{0pt}{}{1}{2}\genfrac{(}{)}{0pt}{}{1}{3}, \quad
\hat{0}:\genfrac{(}{)}{0pt}{}{1}{3}\genfrac{(}{)}{0pt}{}{1}{2}.
\end{align*}
However, the sequence of labels
\begin{align}
\label{eq:1223}
\hat{0}:\genfrac{(}{)}{0pt}{}{1}{2}\genfrac{(}{)}{0pt}{}{2}{3},
\end{align}
is not allowed. 
The non-admissible sequence of labels as in Eq. (\ref{eq:1223}) play a role 
when we define a covering relation for general $k\ge2$.
\end{remark}

Below, we define a covering relation for a general $k\ge2$.
We generalize the construction in the case of $k=1$ to $k\ge2$.

Let $\pi^{(k)}_i\in\mathcal{P}_{n}^{(k)}$, $i=1,2$, and 
$\rho(\pi_2^{(k)})=\rho(\pi_{1}^{(k)})+1$. 
A pair of the covering relation $\pi^{(k)}_1\lessdot\pi^{(k)}_2$ and the label of integers $(\alpha,\beta)_{l}$, $l\in[k]$,
is admissible if they satisfy the following conditions.
\begin{enumerate}[($\diamondsuit$1)]
\item Let $\nu^{(l)}$ be the partition of the $l$-th layer of $\pi^{(k)}_1$.
Then, 
\begin{align*}
\nu^{(l)}:\genfrac{(}{)}{0pt}{}{\alpha}{\beta}_{l}
\end{align*}
is well-defined if it satisfies the following two conditions:  
\begin{enumerate}
\item The points $\alpha$ and $\beta$ belong to different blocks in the first layer (or equivalently 
in the circular presentation of $\pi^{(k)}_{1}$), 
and there is no label $\beta'<\beta$ which belongs to the same block as $\beta$.
\item In the circular presentation of $\nu^{(l)}$, the points $\alpha$ and $\beta$ 
belong to different blocks in $\nu^{(l)}$, and $\alpha$ and $\beta$ are the least labels 
in each block.
\end{enumerate}
In the circular presentation, we connect $\alpha$ with $\beta$ by a line labeled $l$. 
\item We add lines with label $l''\le l$ in such a way that the labels of lines are compatible with 
the lines with labels $l'\ge l$.
\end{enumerate}

\begin{remark}
From the definition of covering relation above, it is clear that the function $\rho$
is a rank function on $\mathcal{P}_{n}^{(k)}$.
From the condition ($\diamondsuit1$), if $\pi\lessdot\pi'$, the number of blocks 
in the first layer is decreased by one. 
As in Remark \ref{remark:rho}, it is enough to focus on the first layer to view 
$\rho$ as a rank function.
\end{remark}

\begin{example}
\begin{enumerate}
\item
The condition ($\diamondsuit2$) implies that 
\begin{align*}
13/24\xrightarrow{(1,2)_l}(12)^l34,
\end{align*}
where $1\le l\le k$.
However, the weighted partition $(14)^{l}(23)^{l'}$ with $2\le l,l'$ cannot 
be obtained from $13/24$ by $(1,2)_{p}$ with some $p$.
\item The chain 
\begin{align*}
13/24:\genfrac{(}{)}{0pt}{}{2}{3}_{l}, \qquad 1\le l\le k,
\end{align*}
is not admissible, since the block containing $13$ in $13/24$ violates 
the condition (a) in ($\diamondsuit1$).
\item Let $(n,k)=(4,2)$ and $\pi_1:=1(23)^2/4$. 
We have a unique chain from $\pi_1$ to $\pi_{2}:=1(234)^2$.
It is 
\begin{align*}
\pi_1\xrightarrow{(2,4)_2}\pi_2,
\end{align*}
To obtain $\pi_2$, we first connect two integers $2$ and $4$ by a line 
with a label $2$. By the condition ($\diamondsuit2$), we connect  two integers 
$3$ and $4$ by a line with a label $2$. 
As a result, we have a block $(234)^{l}$. 
Similarly, the integer $1$ is connected to $2$ and $3$ by a line with a label $1$ in 
$\pi_{1}$. Since the integer $4$ is in the second layer, we have to connect 
$1$ and $4$ by a line with a label $1$.
Then, we obtain $\pi_2=1(234)^2$.

Note that a chain $\pi_1\xrightarrow{(3,4)_2}\pi_2$ is not admissible 
since it violates the condition (b) in ($\diamondsuit1$) (we have a sequence of labels 
identical to Eq. (\ref{eq:1223}) in the second layer).
\end{enumerate}
\end{example}

\begin{example}
\label{ex:nk32}
When $(n,k)=(3,2)$, we have two sublattices which are isomorphic to the lattice in the case
of $(n,k)=(3,1)$. In these two lattices, the layers of the labels of all the edges in the Hasse diagram
are either one or two. They are 
\begin{gather*}
\hat{0}\xrightarrow{(1,2)_k}3/(12)^k\xrightarrow{(1,3)_{k}}(123)^{k}, \quad
\hat{0}\xrightarrow{(1,3)_k}2/(13)^k\xrightarrow{(1,2)_{k}}(123)^{k}, \quad
\hat{0}\xrightarrow{(2,3)_k}(23)^k/1\xrightarrow{(1,2)_{k}}(123)^{k},
\end{gather*}
where $k\in\{1,2\}$.
We have seven maximal chains which consist of labels with different layers.
All these seven maximal chains are as follows:
\begin{gather*}
\hat{0}\xrightarrow{(1,3)_1}13/2\xrightarrow{(1,2)_2}(12)^23, \quad
\hat{0}\xrightarrow{(2,3)_1}1/23\xrightarrow{(1,2)_2}(12)^23, \quad 
\hat{0}\xrightarrow{(1,2)_2}(12)^2/3\xrightarrow{(1,3)_1}(12)^23, \\
%%%%%%%%%%
\hat{0}\xrightarrow{(1,2)_1}12/3\xrightarrow{(1,3)_2}(13)^22, \quad
\hat{0}\xrightarrow{(1,3)_2}(13)^2/2\xrightarrow{(1,2)_1}(13)^22, \\
%%%%%%%%%%
\hat{0}\xrightarrow{(1,2)_{1}}12/3\xrightarrow{(2,3)_{2}}1(23)^2, \quad 
\hat{0}\xrightarrow{(2,3)_{2}}1/(23)^2\xrightarrow{(1,2)_{1}}1(23)^2, 
\end{gather*}
where $\hat{0}=1/2/3$. In total, we have thirteen maximal chains in $\mathcal{P}_{3}^{(2)}$.
\end{example}

The following proposition is an extension of Lemma \ref{lemma:betaperm} 
to the case of $k\ge2$.
\begin{prop}
Let $\pi\in\mathcal{P}^{(k)}_{n}$ be a weighted partition of rank $n-1$.
Consider the chain from the minimum element $1/2/\ldots/n$ to $\pi$.
The label $\beta$ of the labels$ (\alpha,\beta)_{l}$ of the edges in the Hasse diagram 
takes a value in $\{2,3,\ldots,n\}$ exactly once. 
\end{prop}
%%%%%%%%%%%
\begin{proof}
Recall the condition ($\diamondsuit1$) in the definition of the covering relation.
Then, the label $\beta$ takes the value in $\{2,3,\ldots,n\}$ at least once 
since we have to choose $m\in\{2,3,\ldots,n\}$ as $\beta$ to merge two blocks 
into one. 
Similarly, $\beta$ takes the value in $\{2,3,\ldots,n\}$ at most once 
since the value $\beta$ is the maximum of the least element of two blocks 
from Eq. (\ref{eq:alphabeta}). 
From these observations, $\beta$ takes a value in $\{2,3,\ldots,n\}$ exactly once.
\end{proof}

The poset $(\mathcal{P}_{n}^{(k)},\le)$ has the least element 
$\hat{0}=1/2/\ldots/n$.
However, it has several greatest elements which are incomparable with each other.
Recall that the poset is graded by the rank function $\rho$ as in Definition \ref{defn:rk}
There are several weighted partitions whose rank is $n-1$.
We add the element $\hat{1}$ whose rank is $n$ to the poset $(\mathcal{P}_{n}^{(k)},\le)$
to obtain the poset $(\mathcal{L}^{(k)}_{n},\le)$.

\begin{defn}
The covering relation for $\hat{1}\in\mathcal{L}_{n}^{(k)}$ is defined by 
\begin{align*}
\pi^{(k)}\xrightarrow{(1,n)_{k}}\hat{1},
\end{align*}
where $\pi^{(k)}$ satisfies $\rho(\pi^{(k)})=n-1$.
\end{defn}
Note that we have the same labeling on the edges connecting $\pi^{(k)}$'s 
with $\hat{1}$ in the Hasse diagram.
We define the rank of $\hat{1}$ by $\rho(\hat{1})=n$.
This new poset has the least and greatest elements.

\begin{defn}
We denote by $(\mathcal{L}_{n}^{(k)},\le)$ the poset obtained from 
$(\mathcal{P}_{n}^{(k)},\le)$ by adding $\hat{1}$.
\end{defn}

For example, to obtain the lattice $\mathcal{L}_{3}^{(2)}$ from $\mathcal{P}_{3}^{(2)}$, 
we extend the maximal chains in Example \ref{ex:nk32} by adding 
$\xrightarrow{(1,3)_{2}}\hat{1}$ to their right ends. We have $13$ maximal chains from $\hat{0}$ to $\hat{1}$.

By definition, the rank function of $\mathcal{P}_{n}^{(k)}$ takes a value in $[0,n-1]$, 
however, the rank function of $\mathcal{L}_{n}^{(k)}$ takes a value in $[0,n]$.
We will show that the poset $(\mathcal{L}_{n}^{(k)},\le)$ is 
a graded and bounded lattice.

We first briefly recall the properties of a lattice and a poset.
A lattice is a poset where any two elements $x$ and $y$ has 
a join (i.e. least upper bound) $x\vee y$ and 
a meet (i.e. greatest lower bound) $x\wedge y$.
We call an element $x$ a minimal element if there is no element $y$ such that $y<x$.
An {\it atom} is a cover of a minimal element.
A lattice (or poset) is {\it atomistic} if every element is the join 
of a set of some atoms. 
A graded lattice is {\it semimodular} if its rank function $\rho$ satisfies 
$\rho(x)+\rho(y)\ge \rho(x\vee y)+\rho(x\wedge y)$ for any elements $x$ and $y$.
A lattice is called a {\it geometric lattice} if it is a 
finite atomistic semimodular lattice.

Given two elements $x,y\in\mathcal{P}_{n}^{(1)}$, we define a join $x\vee y$ 
and a meet $x\wedge y$ as follows.
Let $n-r$ and $n-r'$ be the ranks of $x$ and $y$ respectively.
Then, $x$ consists of $r$ blocks $B_{i}$, $1\le i\le r$, and 
$y$ consists of $r'$ blocks $B'_{j}$, $1\le j\le r'$.

We say that two blocks $B_{i}$ and $B'_{j}$ are crossing if 
$B_{i}\cap B'_{j}\neq\emptyset$ and denote it by $B_{i}\leftrightarrow B'_{j}$.
More generally, we write $C\leftrightarrow C'$ where $C$ and $C'$ are 
blocks in either $x$ or $y$, if 
there exists a sequence of crossing blocks
\begin{align*}
C=C_{1}\leftrightarrow C_{2}\leftrightarrow C_{3}\leftrightarrow \ldots
\leftrightarrow C_{m}=C',
\end{align*}
where $m\ge2$ and $C_{i}$'s are the blocks in either $x$ or $y$.
Suppose that $B_{i}\leftrightarrow B'_{j}$.
Let $I(i,j)$ be the set of pairs of integers $(i',j')\in[1,r]\times[1,r']$ such that
\begin{align*}
B_{i}\leftrightarrow B'_{j'}\leftrightarrow B_{i'},
\end{align*}
that is, $B_{i}, B_{i'}$ and $B'_{j'}$ are crossing.

By construction of $I(i,j)$, if $(i',j')\in I(i,j)$, then $I(i,j)=I(i',j')$ as a set of 
pairs of integers.
Let $[[i,j]]$ be a representative of $I(i,j)$.
Then, we define 
\begin{align}
\label{eq:defBcup}
B^{\cup}_{[[i,j]]}:=\bigcup_{(i',j')\in I(i,j)}B_{i'}\cup B'_{j'}.
\end{align}
By definition, we have $B_{i'},B'_{j'}\subseteq B^{\cup}_{[[i,j]]}$ if 
$(i',j')\in I(i,j)$.

Similarly, we define the intersection of $B_{i}$ and $B'_{j}$ as 
\begin{align*}
B^{\cap}_{i,j}:=B_{i}\cap B'_{j}.
\end{align*}
By definition, we have $B^{\cap}_{i,j}\subseteq B_{i}$ and $B^{\cap}_{i,j}\subseteq B'_{j}$.

We define a join and a meet for two elements $x$ and $y$ as follows.
\begin{defn}
The join $x\vee y$ is a partition consisting of the non-empty blocks $B^{\cup}_{[[i,j]]}$.
Similarly, the meet $x\wedge y$ is a partition consisting of the non-empty blocks $B^{\cap}_{i,j}$. 
\end{defn}

The definition above implies that $x$ and $y$ are refinements of a join $x\vee y$, 
and a meet $x\wedge y$ is a refinement of $x$ and $y$.

\begin{example}
Let $x=3/245/16$ and $y=456/2/13$ be two partitions in $\mathcal{P}_{6}^{(1)}$.
We put a prime on the elements of the blocks in $y$.
We have a sequence of crossing blocks
\begin{align*}
2'\leftrightarrow 245\leftrightarrow 4'5'6'\leftrightarrow 16 
\leftrightarrow 1'3'\leftrightarrow 3.
\end{align*}
This sequence contains all the integers in $[6]$.
Therefore, we have the join $x\vee y=123456$.

We have three blocks in $x$ and $y$. 
By calculating the intersections of blocks, we obtain the meet $x\wedge y=6/45/3/2/1$.
\end{example}

\begin{lemma}
\label{lemma:Pn1sm}
The poset $\mathcal{P}_{n}^{(1)}$ is semimodular.
\end{lemma}
%%%%%%%%%%%
\begin{proof}
To show $\mathcal{P}_{n}^{(1)}$ is semimodular, it is enough to show that 
$\rho(x)+\rho(y)\ge \rho(x\vee y)+\rho(x\wedge y)$.
By definition of $B^{\cup}_{[[i,j]]}$, the ranks of the join $x\vee y$ and the meet $x\wedge y$
are given by 
\begin{align*}
\rho(x\vee y)&=n-\#\{[[i,j]] | B^{\cup}_{[[i,j]]}\neq\emptyset\}, \\
\rho(x\wedge y)&=n-\#\{(i,j) | B^{\cap}_{i,j}\neq\emptyset\}.
\end{align*}
From Eq. (\ref{eq:defBcup}) and the construction of $B^{\cup}_{[[i,j]]}$, we have 
\begin{align}
\label{eq:BcupI}
\#\{i' | (i',j')\in I(i,j)\}+\#\{j' | (i',j')\in I(i,j)\}
\le 1+ \#\{(i',j')\in I(i,j)| B^{\cap}_{i',j'}\neq\emptyset \},
\end{align}
for any pair $[[i,j]]\in[1,r]\times[1,r']$ if $B_{i}$ and $B'_{j}$ are crossing.
The left hand side of Eq. (\ref{eq:BcupI}) simply counts the numbers of blocks which
appear in the sequence of crossing blocks containing $B_{i}$ and $B'_{j}$.
If the blocks $B_{i'}$ and $B'_{j'}$ are crossing, we have $B^{\cap}_{i'j'}\neq\emptyset$.
If a sequence of crossing blocks does not contain a loop, the equality holds 
in Eq. (\ref{eq:BcupI}). However, if a sequence contains a loop, the equality does not hold.

Note that we have
\begin{align}
\label{eq:BiBj1}
&\sum_{[[i,j]]}\#\{i' | (i',j')\in I(i,j)\}=r, \\
\label{eq:BiBj2}
&\sum_{[[i,j]]}\#\{j' | (i',j')\in I(i,j)\}=r', \\
\label{eq:BiBj3}
&\sum_{[[i,j]]}1=\#\{[[i,j]] | B^{\cup}_{[[i,j]]}\neq\emptyset \}, \\
\label{eq:BiBj4}
&\sum_{[[i,j]]}\#\{(i',j')\in I(i,j)| B^{\cap}_{i',j'}\neq\emptyset \}
=\#\{(i',j')| B^{\cap}_{i',j'}\neq\emptyset \}.
\end{align}
Equation (\ref{eq:BiBj1}) (resp. Eq. (\ref{eq:BiBj2})) follows from the fact that the left hand side 
simply counts the number of blocks in $x$ (resp. $y$).
Equation (\ref{eq:BiBj3}) counts the number of representatives $[[i,j]]$ which is equal the number of 
blocks in $x\vee y$.
Similarly, Eq. (\ref{eq:BiBj4}) counts the number of blocks in $x\wedge y$.
From these observations, we have 
\begin{align*}
\rho(x\vee y)+\rho(x\wedge y)
&=2n-\#\{[[i,j]] | B^{\cup}_{[[i,j]]}\neq\emptyset\}-\#\{(i,j) | B^{\cap}_{i,j}\neq\emptyset\}, \\
&=2n-\sum_{[[i,j]]}\left(1+\#\{(i',j')\in I(i,j)|B_{i',j'}^{\cap}\neq\emptyset\}\right), \\
&\le 2n-\sum_{[[i,j]]}\left(\#\{i'|(i',j')\in I(i,j)\}+\#\{j'|(i',j')\in I(i,j)\}\right), \\	
&=2n-r-r', \\
&=\rho(x)+\rho(y), 
\end{align*}
where we have used Eq. (\ref{eq:BcupI}) in the second line.
This completes the proof.
\end{proof}

\begin{example}
We consider Eq. (\ref{eq:BcupI}) for $x=12/34$ and $y=13/24$.
We have a following sequence of crossing blocks for $x$ and $y$:
\begin{align*}
\tikzpic{-0.5}{
\node(1)at(0,0){$12$};
\node(2)at(1.2,0){$1'3'$};
\node(3)at(0,-1.2){$2'4'$};
\node(4)at(1.2,-1.2){$34$};
\draw[<->](1)--(2);\draw[<->](1)--(3);\draw[<->](2)--(4);\draw[<->](3)--(4);
}
\end{align*}
where a block consisting of integers (resp. primed integers) is a block of $x$ (resp. $y$).
In this case, we have $\#\{i' | (i',j')\in I(i,j)\}+\#\{j' | (i',j')\in I(i,j)\}=4$,
and $1+ \#\{(i',j')\in I(i,j)| B^{\cap}_{i',j'}\neq\emptyset \}=5$.
\end{example}

\begin{example}
Let $x=12/34$ and $y=13/24$.
The ranks of $x$ and $y$ are given by $\rho(x)=\rho(y)=2$.
Then, we have $x\vee y=1234$ and $x\wedge y=1/2/3/4$, and 
$\rho(x\vee y)=3$ and $\rho(x\wedge y)=0$.
This implies $4=\rho(x)+\rho(y)>\rho(x\vee y)+\rho(x\wedge y)=3$.
\end{example}

Fix $k\ge2$ and let $x,y\in\mathcal{P}_{n}^{(k)}$.
Given $x$, we denote by $x^{(l)}$, $1\le l\le k$, the weighted partition in the $l$-th 
layer of $x$. Note that $x^{(l)}$ is isomorphic to a partition in $\mathcal{P}_{n}^{(1)}$.
We define 
\begin{align*}
(x\vee y)^{(l)}:=x^{(l)}\vee y^{(l)}, \qquad
(x\wedge y)^{(l)}:=x^{(l)}\wedge y^{(l)},
\end{align*}
where $1\le l\le k$.

\begin{lemma}
\label{lemma:Pnksm}
The poset $\mathcal{P}_{n}^{(k)}$ is semimodular.
\end{lemma}
%%%%%%%%%%%
\begin{proof}
The rank $\rho(x)$ of $x\in\mathcal{P}_{n}^{(k)}$ is equal to 
$\rho(x^{(1)})$ where $x^{(1)}$ is the partition of the first layer 
in the weighted partition $x$.
By Lemma \ref{lemma:Pn1sm}, it is easy to see that $\mathcal{P}_{n}^{(k)}$
is semimodular.
\end{proof}

The atoms $\nu\in\mathcal{P}_{n}^{(k)}$ is a weighted partition such that 
the points $i$ and $j$, $1\le i<j\le n$, 
are connected by a line with a label $1\le l\le k$ in the circular presentation.
We denote by $\nu=(i,j)_{k}$ the weighted partition $\nu$.
Since $\nu$ has a single line connecting points, the rank $\rho(\nu)=1$.
There are $kn(n-1)/2$ atoms in $\mathcal{P}_{n}^{(k)}$.

We summarize the properties of the posets $(\mathcal{P}_{n}^{(k)},\le)$ and 
$(\mathcal{L}_{n}^{(k)},\le)$.
\begin{prop}
\label{prop:PL}
\begin{enumerate}
\item
The poset $(\mathcal{P}_{n}^{(k)},\le)$ is a geometric lattice.
\item 
The poset $(\mathcal{L}_{n}^{(k)},\le)$ is a bounded graded lattice.
\end{enumerate}
\end{prop}
%%%%%%%%%%
\begin{proof}
(1). It is obvious that the poset $(\mathcal{P}_{n}^{(k)})$ is finite.
From Lemma \ref{lemma:Pnksm}, it is semimodular.
Let $x\in\mathcal{P}_{n}^{(k)}$. In the circular presentation of $x$, 
suppose that the points $i$ and $j$, $1\le i<j\le n$, belong to the 
$l(i,j)$-th layer. 
Let $P(x)$ be the set of pairs of points $(i,j)$ such that 
two points $i$ and $j$ are connected by a line in the $l(i,j)$-th layer in $x$.
Then, $x$ is expressed in terms of atoms:
\begin{align*}
x=\bigvee\limits_{(i,j)\in P(x)}(i,j)_{l(i,j)}.
\end{align*}
This implies that $\mathcal{P}_{n}^{(k)}$ is atomistic.
From these, $\mathcal{P}_{n}^{(k)}$ is a geometric lattice.

(2). Note that $\mathcal{L}_{n}^{(k)}$ is obtained from $\mathcal{P}_{n}^{(k)}$
by adding the maximum element $\hat{1}$.
From (1), $\mathcal{L}_{n}^{(k)}$ is a bounded graded lattice with 
the rank function $\rho$.
This completes the proof.
\end{proof}

\begin{remark}
The poset $\mathcal{L}_{n}^{(k)}$ is not geometric since the greatest element $\hat{1}$
cannot be obtained as a join of atoms.
\end{remark}

\section{\texorpdfstring{$EL$}{EL}-labeling, M\"obius function and order homology}
\label{sec:EL}
We briefly recall the definition of the $EL$-labeling on the Hasse diagram 
of a lattice and some basic properties following \cite{Bjo80,BjoGarSta82}.	

We first introduce $EL$-labeling of a lattice.
For any finite poset $P$, we denote its covering relation 
by $C(P):=\{(x,y)\in P\times P| x\lessdot y\}$.
Then, an {\it edge-labeling} of $P$ is a map 
$\lambda:C(P)\rightarrow \Lambda$ where $\Lambda$ is a poset.
We consider a unrefinable chain 
$c:x_{0}\lessdot x_{1}\lessdot\ldots\lessdot x_{n}$.
An edge-labeling $\lambda$ is called {\it rising}
if $\lambda(x_{i-1},x_i)\le \lambda(x_i,x_{i+1})$  for
$1\le i\le n-1$.

\begin{defn}
\label{defn:EL}
Let $\lambda:C(P)\rightarrow\Lambda$ be an edge-labeling of a graded 
poset $P$.
The edge-labeling $\lambda$ is said to be an {\it $R$-labeling} 
if it satisfies the following condition:
\begin{enumerate}[(a)]
\item For any interval $[x,y]$ of $P$, there exists a unique 
rising unrefinable chain $c$.
\end{enumerate}
The edge-labeling $\lambda$ is said to be an {\it $EL$-labeling} if 
it satisfies 
\begin{enumerate}[(i)]
\item $\lambda$ is an $R$-labeling.
\item Suppose that $c:x=x_{0}\lessdot x_{1}\lessdot\ldots\lessdot x_{n}=y$ 
is the unique rising unrefinable chain for an interval $[x,y]$.
If $x\lessdot z\le y$, $z\neq x_{1}$, then $\lambda(x,x_1)<\lambda(x,z)$.
\end{enumerate}
\end{defn}

\begin{defn}
A poset is said to be {\it lexicographically shellable} (or $EL$-shellable)
if it is graded and admits an $EL$-labeling.
\end{defn}

One of the main results in \cite{Bjo80} is the following theorem.
\begin{theorem}[Theorem 2.3 in \cite{Bjo80}]
\label{thrm:shellable}
Let $P$ be a lexicographically shellable poset. 
Then, the poset $P$ is shellable.
\end{theorem}
%%%%%%%%%%%%%%

To apply Theorem \ref{thrm:shellable} to $\mathcal{L}_{n}^{(k)}$, 
we have to show there exists an $EL$-labeling on $\mathcal{L}_{n}^{(k)}$.
Recall that we write the covering relation as $\pi_1\xrightarrow{(\alpha,\beta)_{l}}\pi_2$, 
and we call $(\alpha,\beta)_{l}$ a label.
We introduce a linear order on the labels $(\alpha,\beta)_{l}$ of the edges in the Hasse diagram.
\begin{align}
\label{eq:deflex}
\genfrac{(}{)}{0pt}{}{\alpha_1}{\beta_2}_{l}<\genfrac{(}{)}{0pt}{}{\alpha_2}{\beta_2}_{l'}
\quad
\text{if}\ 
\begin{cases}
l>l', \\
\alpha_1<\alpha_2 \text{ if } l=l', \\ 
\beta_1<\beta_2 \text{ if  } l=l' \text{ and } \alpha_{1}=\alpha_2.
\end{cases}
\end{align}
This means that we choose a poset $\Lambda$ as a totally ordered set $\mathbb{N}\times\mathbb{N}$
in Definition \ref{defn:EL}.
For example, in the case of $(n,k)=(3,2)$, we have six labels and 
they satisfy 
\begin{align*}
\genfrac{(}{)}{0pt}{}{1}{2}_2<\genfrac{(}{)}{0pt}{}{1}{3}_2
<\genfrac{(}{)}{0pt}{}{2}{3}_2<\genfrac{(}{)}{0pt}{}{1}{2}_1
<\genfrac{(}{)}{0pt}{}{1}{3}_1<\genfrac{(}{)}{0pt}{}{2}{3}_1.
\end{align*}

\begin{theorem}
The labeling given in Section \ref{sec:cr} is an $EL$-labeling 
of the lattice $\mathcal{L}_{n}^{(k)}$.
\end{theorem}
%%%%%%%%%%%%%
\begin{proof}
We have to show that in any interval $[x,y]$ of $\mathcal{L}_{n}^{(k)}$, 
there is a unique rising chain.
Further, to show the labeling is an $EL$-labeling, this unique rising chain has to be 
lexicographically first.
We have four cases:
\begin{enumerate}
\item Intervals $[\hat{0},\hat{1}]$.
Recall that we have $\pi\xrightarrow{(1,n)_{k}}\hat{1}$ for any 
$\pi$ satisfying $\rho(\pi)=n-1$.
When we merge two blocks into one, the integer $\alpha$ in the label 
$(\alpha,\beta)_{l}$ is the least element in these two blocks.
Any rising maximal chain should have a label word
\begin{align}
\label{eq:0to1}
\hat{0}:\genfrac{(}{)}{0pt}{}{1}{2}_{k}\genfrac{(}{)}{0pt}{}{1}{3}_{k}\ldots
\genfrac{(}{)}{0pt}{}{1}{n}_{k}\genfrac{(}{)}{0pt}{}{1}{n}_{k}.
\end{align}
By construction of a label word, any chain except the maximal chain (\ref{eq:0to1}) 
contains a label $(\alpha,\beta)_{l}>(1,n)_{k}$, or the chain is not increasing.
It is obvious that the chain (\ref{eq:0to1}) is a unique rising chain and lexicographically first.

\item Interval $[\pi,\hat{1}]$ for $\rho(\pi)=n-1$.
Since $\rho(\pi)=n-1$,  we have $\pi\xrightarrow{(1,n)_{k}}\hat{1}$.
The label $(1,n)_{k}$ is a unique rising label and lexicographically first.

\item  Intervals $[\pi,\hat{1}]$ for $\rho(\pi)<n-1$.
Since $\rho(\pi)<n-1$, we have to merge blocks to obtain $\pi'$ 
such that $\pi'\in[\pi,\hat{1}]$ and $\rho(\pi')=n-1$.
When we merge blocks in $\pi$, we can choose a label $(1,n_1)_{k}$ 
with some $n_1$ by merging two blocks one of which contains the point $1$.
By choosing the minimum integer $n_1$ in this way, we have 
$\pi\xrightarrow{(1,n_1)_{k}}\pi_1$.
By applying the same procedure to $\pi_1$, we have a sequence of labels:
\begin{align}
\label{eq:intpi1}
\pi\xrightarrow{(1,n_1)_{k}}\pi_1\xrightarrow{(1,n_2)_{k}}\ldots
\xrightarrow{(1,n_m)_{k}}\pi'\xrightarrow{(1,n)_{k}}\hat{1},
\end{align}
where $m=n-1-\rho(\pi)$.
Note that one can choose the integers $n_1,\ldots,n_{m}$ such that 
$n_1<n_2<\ldots<n_{m}$.
Further, we have $n_m\le n$.
The sequence (\ref{eq:intpi1}) is rising and unique.
By construction of $n_1$, the sequence $(\ref{eq:intpi1})$ 
satisfies the condition (ii) in Definition \ref{defn:EL}, 
which implies it is lexicographically first. 

\item Intervals of the form $[\pi,\nu]$ with 
$\hat{0}\le \pi$ and $\nu\neq\hat{1}$.
Since $\pi\le\nu$, the set of edges in $\nu$ contains 
all edges in $\pi$.
Let $E(\pi)$ be the set of labeled edges in a weighted partition $\pi$.
We can pick up an edge $e\in E(\nu)\setminus E(\pi)$ such that 
$e$ is in the maximum layer labeled $l_1$, and the pair of points $(\alpha_1,\beta_1)$
is minimum in the linear order.
Let $\pi_1$ be a weighted partition satisfying $\pi\xrightarrow{(\alpha_1,\beta_1)_{l_1}}\pi_1$.
Then, we repeat the same procedure for $[\pi_1,\nu]$, 
we have a unique rising labels:
\begin{align*}
\pi\xrightarrow{(\alpha_1,\beta_1)_{l_1}}\pi_1\xrightarrow{(\alpha_2,\beta_2)_{l_2}}\pi_2
\ldots \xrightarrow{(\alpha_m,\beta_{m})_{l_m}}\nu,
\end{align*}
where 
\begin{align*}
\genfrac{(}{)}{0pt}{}{\alpha_1}{\beta_1}_{l_1}<\genfrac{(}{)}{0pt}{}{\alpha_2}{\beta_2}_{l_2}
<\ldots<\genfrac{(}{)}{0pt}{}{\alpha_m}{\beta_m}_{l_m}.
\end{align*}
By the construction of $(\alpha_1,\beta_1)_{l_1}$,
this label $(\alpha_1,\beta_{1})_{l_1}$
is the lexicographically first.
\end{enumerate}
In above four cases, any interval $[x,y]$ contains a unique rising chain $c$ 
and $c$ is lexicographically first.
This completes the proof.
\end{proof}

The lexicographically shellability of a poset $P$ gives 
a combinatorial interpretation of the M\"obius function.
We recall the definition of M\"obius function. 
\begin{defn}
Let $\mathcal{P}$ be a poset.
The M\"obius function $\mu:\mathcal{P}\times \mathcal{P}\rightarrow\mathbb{Z}$ is 
recursively defined by
\begin{align*}
\mu(x,y):=
\begin{cases}
1, & \text{ if } x=y, \\
-\sum_{x\le z<y}\mu(x,z), & \text{ if } x<y.
\end{cases}
\end{align*}
Especially, we define $\mu(\mathcal{P}):=\mu(\hat{0},\hat{1})$ if $\hat{0}$ and 
$\hat{1}$ exist.
\end{defn}

\begin{remark}
The M\"obius function of a poset is a combinatorial invariant and 
gives useful information of the structure of the poset.
Another important invariant is the zeta polynomials of the poset, 
which enumerate multichains. Some interesting examples of the computations
of zeta polynomials are given in \cite{Kre65,Kre72}.
These two functions, the M\"obius function and the zeta polynomials, 
are inverses of each other \cite{Ede80a}.
As we will see below, we interpret the M\"obius function in terms 
of enumeration of the chains (see Proposition \ref{prop:Mobiusdc}).
\end{remark}

The following proposition is a key to compute the M\"obius 
function $\mu(\mathcal{P})$ of the poset $\mathcal{P}$, when
the poset admits an $EL$-labeling.

\begin{prop}[\cite{Bjo80,Sta74}]
\label{prop:Mobiusdc}
Suppose the poset $\mathcal{P}$ admits an $EL$-labeling.
The M\"obius function $\mu(\mathcal{P})$ is equal to the number of 
maximal decreasing chains from $\hat{0}$ to $\hat{1}$.
\end{prop}
We will calculate the M\"obius function of the lattice $\mathcal{L}_{n}^{(k)}$
by use of Proposition \ref{prop:Mobiusdc} in Section \ref{sec:CM}.

Let $P$ be a finite poset, and $\mathsf{k}$ be a field or the ring of integers $\mathbb{Z}$, and 
$\tilde{H}_{i}(P,\mathsf{k})$ be the {\it order homology groups of $P$ with coefficients in $\mathsf{k}$}.
By use of the {\it Euler--Poincar\'e formula}, 
the M\"obius function $\mu(P)$ is the {\it Euler characteristic} of order homology:
\begin{align*}
\mu(P)=\sum_{i=-1}^{l(P)}(-1)^{i}\dim_{\mathsf{k}}\tilde{H}_{i}(P,\mathsf{k}).
\end{align*}

\begin{defn}
A graded poset $P$ is said to be {\it Cohen--Macaulay over $\mathsf{k}$} 
if for each open interval $(x,y)$ in $P$ satisfies
\begin{align*}
\tilde{H}_{i}((x,y),\mathsf{k})=0, \text{ for } i\neq \rho(x,y)-2.
\end{align*}
\end{defn}
Note that $\rho(x,y)-2$ is the dimension of a simplicial complex $\Delta((x,y))$.

\begin{theorem}[Theorem 3.2 in \cite{Bjo80}]
\label{thrm:lsCM}
If $P$ is lexicographically shellable, then $P$ is Cohen--Macaulay over 
$\mathbb{Z}$ and all fields $\mathsf{k}$.
\end{theorem}

Recall $\mathcal{L}:=\mathcal{L}_{n}^{(k)}$ is lexicographically shellable.
Thus, from Theorem \ref{thrm:lsCM}, we have 
\begin{align*}
(-1)^{\rho(x,y)}\mu(x,y)=\dim_{\mathsf{k}}\tilde{H}_{\rho(x,y)-2}((x,y),\mathsf{k})\ge0,
\end{align*}
where $(x,y)$ is an open interval of $\mathcal{L}$.
Especially, if $(x,y)=(\hat{0},\hat{1})$, we have $\rho(x,y)=n$, and 
\begin{align*}
(-1)^{n}\mu(\mathcal{L})=\dim_{\mathsf{k}}\tilde{H}_{n-2}(\mathcal{L},\mathsf{k}).
\end{align*}

\section{Computation of the M\"obius function}
\label{sec:CM}
\subsection{M\"obius function}
We consider the lattice $\mathcal{L}^{(k)}_{n}$ to compute its M\"obius function. 
For each permutation $\sigma'$ in the set $\{2,3,\ldots,n\}$, we assign a circular diagram as follows.
We append $1$ to $\sigma'$ from left and obtain a permutation $\sigma$ in $[n]$.
In the one-line notation of $\sigma:=(\sigma_1,\ldots,\sigma_{n})$, we construct following pairs of integers.
Given $2\le j\le n$, let $i$ be a maximum integer such that $i<j$ and $\sigma_{i}<\sigma_{j}$.
Note that there is no $i'$ satisfying $i<i'<j$ and $\sigma_{i'}<\sigma_j$.
Since $\sigma_1=1$, this operation is well-defined. 
Let $I(\sigma)$ be the set of pairs $(\sigma_{i},\sigma_{j})$ constructed as above.
In the circular presentation, we connect the points labeled $i$ with the points labeled $j$ if
$(i,j)\in I(\sigma)$.

\begin{lemma}
\label{lemma:k1dc}
Let $\sigma$ be a permutation in $[n]$ and $\sigma_1=1$.
A circular presentation of $I(\sigma)$ is bijective to a 
maximal decreasing chain of $\mathcal{P}^{(1)}_{n}$. 
\end{lemma}
%%%%%%%%%%
\begin{proof}
By definition of the circular representation of $I(\sigma)$, 
each integer $j\in\{2,3,\ldots,n\}$ is connected to an integer $i\in\{1,2,\ldots,i-1\}$.
Further, integers in $[2,n]$ appear exactly once as $j$.
We sort the pairs $(i,j)\in I(\sigma)$ in the reverse lexicographic order, and 
obtain a decreasing chain.
Then, it is obvious that the sequence of the pairs avoids the 
pattern Eq. (\ref{eq:1223}).
The number of labels is $n-1$ and this implies that the sequence is a 
maximal chain in $\mathcal{P}_{n}^{(1)}$.
From these observations, $I(\sigma)$ is bijective to a maximal decreasing chain 
of $\mathcal{P}^{(1)}_{n}$.
\end{proof}

\begin{example}
We consider $I(145362)=\{(2,1),(6,3),(3,1),(5,4),(4,1)\}$. 
Its circular presentation and the corresponding maximal decreasing chain
are given by
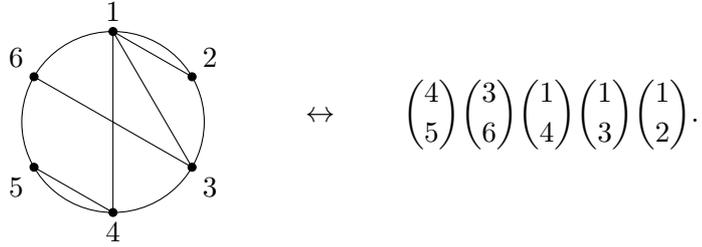
\begin{align*}
\tikzpic{-0.5}{
\draw circle(1.2cm);
\foreach \a in {90,30,-30,-90,-150,-210}
\filldraw [black] (\a:1.2cm)circle(1.5pt);
\draw (90:1.2cm) node[anchor=south] {$1$};
\draw (30:1.2cm) node[anchor=south west] {$2$};
\draw (-30:1.2cm) node[anchor=north west] {$3$};
\draw (-90:1.2cm) node[anchor=north] {$4$};
\draw (-150:1.2cm) node[anchor=north east] {$5$};
\draw (-210:1.2cm) node[anchor=south east] {$6$};
\draw(30:1.2cm)to(90:1.2cm)(-210:1.2cm)to(-30:1.2cm)to(90:1.2cm)
(-150:1.2cm)to(-90:1.2cm)to(90:1.2cm);
}
\qquad\leftrightarrow\qquad 
\genfrac{(}{)}{0pt}{}{4}{5}\genfrac{(}{)}{0pt}{}{3}{6}\genfrac{(}{)}{0pt}{}{1}{4}
\genfrac{(}{)}{0pt}{}{1}{3}\genfrac{(}{)}{0pt}{}{1}{2}.
\end{align*}
Note that we have a unique maximal decreasing chain from $I(145362)$.
\end{example}

\begin{remark}
The circular diagram for $I(\sigma)$ is essentially the same 
as a diagram in $\mathcal{C}(n,1)$ introduced in Definition \ref{defn:Cnr}.
\end{remark}

Let $\tau$ be the circular presentation of $I(\sigma)$.
We consider the coloring of the edges in $\tau$ as follows:
\begin{enumerate}
\item An edge connecting the points $1$ and $i$, $2\le i$, 
has a color $c\in\{1,2,\ldots,k-1\}$, 
\item An edge connecting the points $i$ and $j$, $1<i<j\le n$, 
has a color $c'\in\{1,2,\ldots,k\}$.
\end{enumerate}
We denote by $\mathbf{c}$ the coloring of the circular presentation 
$\tau$ which satisfies the above two conditions.

Lemma \ref{lemma:k1dc} is generalized to the case of $k\ge2$:
\begin{lemma}
\label{lemma:taucmdc}
Let $k\ge2$.
A circular presentation $\tau$ with a coloring $\mathbf{c}$ is 
bijective to a maximal decreasing chain in the Hasse diagram.
\end{lemma}
%%%%%%%%%%%%
\begin{proof}
Recall the order of labels defined in Eq. (\ref{eq:deflex}).
As in the proof of Lemma \ref{lemma:k1dc}, we sort the labels in $I(\sigma)$ with 
the coloring $\mathbf{c}$ in the reverse lexicographic order.
For $(i,j)\in I(\sigma)$ with a color $c$, we associate the label 
$(i,j)_{c}$ to it.
It is obvious, by construction, that the chain of labels satisfy the conditions ($\diamondsuit1$) 
in Section \ref{sec:cr}.
Then, $\tau$ with the color $\mathbf{c}$ is bijective to a maximal decreasing chain 
in $\mathcal{L}_{n}^{(k)}$.
\end{proof}

Recall that the label from $\pi$ with $\rho(\pi)=n-1$ to $\hat{1}$ is $(1,n)_{k}$.
The conditions of a coloring $\mathbf{c}$ are compatible with this label $(1,n)_{k}$.
In other words, the decreasing chain from $\hat{0}$ to  $\pi\lessdot\hat{1}$ should 
satisfy that the last label is larger than $(1,n)_{k}$.

\begin{defn}
We define the exponential generating function for the M\"obius functions 
of the lattice $\mathcal{L}_{n+1}^{(k)}$ by 
\begin{align*}
\Gamma^{(k)}(x):=
\sum_{0\le n}(-1)^{n+1}\mu(\mathcal{L}_{n+1}^{(k)})\genfrac{}{}{}{}{x^{n}}{n!}.
\end{align*}
\end{defn}
First few terms of $\Gamma^{(k)}(x)$ are 
\begin{align*}
\Gamma^{(2)}(x)&=1+x+\genfrac{}{}{}{}{3x^2}{2!}+\genfrac{}{}{}{}{15x^3}{3!}
+\genfrac{}{}{}{}{105x^4}{4!}+\cdots, \\
\Gamma^{(3)}(x)&=1+2x+\genfrac{}{}{}{}{10x^2}{2!}+\genfrac{}{}{}{}{80x^3}{3!}
+\genfrac{}{}{}{}{880x^4}{4!}+\cdots, \\
\Gamma^{(4)}(x)&=1+3x+\genfrac{}{}{}{}{21x^2}{2!}+\genfrac{}{}{}{}{231x^3}{3!}
+\genfrac{}{}{}{}{3465x^4}{4!}+\cdots, \\
\Gamma^{(5)}(x)&=1+4x+\genfrac{}{}{}{}{36x^2}{2!}+\genfrac{}{}{}{}{504x^3}{3!}
+\genfrac{}{}{}{}{9576x^4}{4!}+\cdots. 
\end{align*}
These coefficients corresponds to the sequence A001147, A008544, A008545, and A008546
respectively in OEIS \cite{Slo}.

The next theorem is one of the main results of this paper.
\begin{theorem}
\label{thrm:Gamma}
The exponential generating function $\Gamma^{(k)}(x)$ of the M\"obius function $\mu(\mathcal{L}_{n}^{(k)})$ is given by 
\begin{align}
\label{eq:Gamma}
\Gamma^{(k)}(x)=(1-kx)^{-(k-1)/k}.
\end{align}
\end{theorem}
%%%%%%%%%%%%%%
\begin{proof}
From Proposition \ref{prop:Mobiusdc} and Lemma \ref{lemma:taucmdc}, the value $\mu(\mathcal{L}^{(k)}_{n+1})$ is 
equal to the number of $\tau\in I(\sigma)$ with possible coloring $\mathbf{c}$.
The circular presentation of $\tau$, the point $j>1$ can be connected to the point $i<j$.
By definition of the coloring, we have $k$ ways to put  a color on the points $i$ and $j$ 
with $i\neq1$, and $k-1$ ways for $i=1$.
Then, we have $(k(j-2)+k-1)$ possible ways to choose a point $i\in[1,j-1]$ and its coloring.
Therefore, we have 
\begin{align*}
\mu(\mathcal{L}_{n+1}^{(k)})&=(-1)^{n-1}\prod_{j=2}^{n+1}(k(j-2)+k-1), \\
&=(-1)^{n-1}\prod_{j=0}^{n-1}(k(j+1)-1).
\end{align*}
It is straightforward to see the exponential generating function $\Gamma^{(k)}(x)$ 
is given by Eq. (\ref{eq:Gamma}).
\end{proof}

\subsection{Maximal decreasing chains and labeled binary trees}
In this subsection, we introduce the notion of a labeled binary tree 
which is bijective to a maximal decreasing chain of $\mathcal{L}_{n}^{(k)}$.
This notion is a generalization of labeled binary trees studied in \cite{Bar90,BarBer90}.

A {\it complete binary tree} is a rooted binary tree such that each internal node 
has exactly two children.
A {\it leaf} is an external node which has no children.

\begin{defn}
We denote by $\mathcal{T}_{n}$ the set of complete binary trees with $n$ leaves.
\end{defn}

It is well-known that the total number of complete binary trees in $\mathcal{T}_{n}$ is given by 
the $n-1$-th Catalan number, {\it i.e.}, 
\begin{align*}
\left|\mathcal{T}_{n}\right|
=\genfrac{}{}{}{}{1}{n}\genfrac{(}{)}{0pt}{}{2(n-1)}{n-1},
\end{align*}
for $n\ge2$.

We introduce a liner order for integer labels with a subscript, that is, 
$l_{q}<l'_{p}$ if $q>p$ or $l<l'$ if $p=q$:
\begin{align}
\label{eq:linord}
1_{k}<2_{k}<\ldots<n_{k}<1_{k-1}<\ldots<n_{k-1}<\ldots<1_{1}<\ldots<n_{1}. 
\end{align}
Note that the linear order (\ref{eq:linord}) is compatible with the 
linear order (\ref{eq:deflex}).

A {\it labeled binary tree} $L(T)$ of shape $T\in\mathcal{T}_{n}$ is a binary tree $T$ 
such that it has labels on the leaves and internal nodes (including the root).
A label consists of an integer in $[n]$ and a subscript in $[k]$, and it 
satisfies the linear order (\ref{eq:linord}).

We first construct a labeled binary tree for $k=1$.
Let $\sigma:=\sigma_1\cdots\sigma_{n}$ be a permutation of the set $\{1,2,\ldots,n\}$
and satisfying $\sigma_1=1$.
We define a non-negative integer $s_{j}$, $j\in[n]$, by 
\begin{align*}
s_{j}:=\#\{i | \sigma_{i}<\sigma_{j}, i<j \}.
\end{align*}
A labeled binary tree $T(\sigma)$ associated to $\sigma$ is recursively defined as follows:
\begin{enumerate}[($\clubsuit1$)]
\item $T(\sigma)$ consists of a single node, the root, if $\sigma=\sigma_1=1$.
\item $T(\sigma)$ consists of the root, and its left and right child if $\sigma=\sigma_1\sigma_2=12$.
The label of the left child is $1$ and that of the right child is $2$.
There is no label on the root.
\item Let $\sigma':=\sigma'_{1}\cdots\sigma'_{n-1}$ be a permutation on $[1,n-1]$ and $T(\sigma')$ be a binary tree
associated to $\sigma'$.
The permutation $\sigma'$ can be obtained from $\sigma$ by deleting $n$.
Let $\sigma_{l}=n$.
Then, $T(\sigma)$ is obtained from $T(\sigma')$ by attaching a left and right child
at the $s_{l}$-th leaf of $T(\sigma')$ from left. 
The label of the newly added left leaf is equal to that of its parent, and 
the label of the newly added right leaf is equal to $n$.
\end{enumerate}

By definition of $\sigma$ and $s_{j}$, we always have $s_{j}\ge1$. The addition 
of leaves to $T(\sigma')$ in ($\clubsuit3$) is well-defined.

Note that we have $n-1$ internal nodes and $n$ external nodes (leaves) in 
the tree $T(\sigma)$.
If we read the labels of the leaves of $T(\sigma)$ from left to right,
the word is nothing but a permutation $w:=1\sigma_2\ldots\sigma_{n}$ in 
one-line notation.

Suppose that $\sigma$ is a permutation of length $n$ with $\sigma_1=1$.
Since the total number of labeled binary trees is equal to the number of 
permutations $\sigma$ of length $n$ with $\sigma_1=1$, we have 
$\#\{\sigma\in\mathbb{S}_{n}: \sigma_1=1\}=(n-1)!$ where $\mathbb{S}_{n}$ 
is the symmetric group of $[n]$.
This number $(n-1)!$ is the same as the signless M\"obius function of the lattice 
of partitions.
The above recursive definition of $T(\sigma)$ implies that
$\#\{T(\sigma): \sigma\in\mathbb{S}_{n}, \sigma_1=1\}=(n-1)!$.

\begin{example}
\label{ex:LBTk1}
When $n=4$, we have $3!=6$ labeled binary trees:
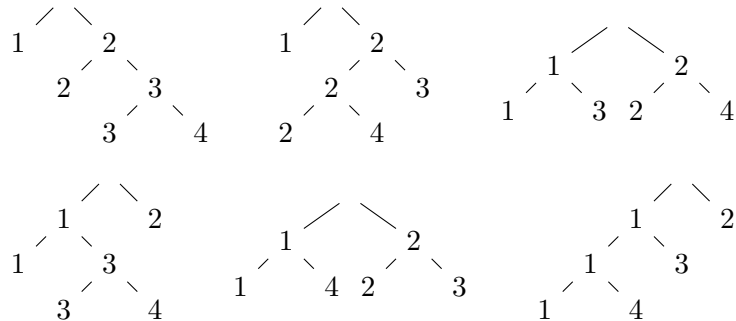
\begin{align*}
\tikzpic{-0.5}{[scale=0.6]
\node (root)at(0,0){};
\node (L)at(-1,-1){$1$};
\node (R)at(1,-1){$2$};
\node (RL)at(0,-2){$2$};
\node (RR)at(2,-2){$3$};
\node (RRL)at(1,-3){$3$};
\node(RRR)at(3,-3){$4$};
\draw(root)--(L)(root)--(R)--(RL)(R)--(RR)--(RRL)(RR)--(RRR);
} \quad
\tikzpic{-0.5}{[scale=0.6]
\node (root)at (0,0){};
\node (L)at (-1,-1){$1$};
\node (R) at (1,-1){$2$};
\node(RL) at (0,-2){$2$};
\node(RR)at (2,-2){$3$};
\node(RLL)at (-1,-3){$2$};
\node(RLR)at(1,-3){$4$};  
\draw(root)--(L)(root)--(R)--(RL)--(RLL)(R)--(RR)(RL)--(RLR);
} \quad
\tikzpic{-0.5}{[scale=0.6]
\node (root)at(0,0){};
\node (L)at (-1.4,-1){$1$};
\node (R) at(1.4,-1){$2$};
\node (LL)at(-2.4,-2){$1$};
\node (LR)at(-0.4,-2){$3$};
\node (RL)at (0.4,-2){$2$};
\node (RR)at(2.4,-2){$4$};
\draw(root)--(L)--(LL)(L)--(LR)(root)--(R)--(RL)(R)--(RR);
} \\
\tikzpic{-0.5}{[scale=0.6]
\node(root)at(0,0){};
\node(L)at(-1,-1){$1$};
\node(R)at(1,-1){$2$};
\node(LL)at(-2,-2){$1$};
\node(LR)at(0,-2){$3$};
\node(LRL)at(-1,-3){$3$};
\node(LRR)at(1,-3){$4$};
\draw(root)--(L)--(LL)(L)--(LR)--(LRL)(LR)--(LRR)(root)--(R);
}\quad
\tikzpic{-0.5}{[scale=0.6]
\node (root)at(0,0){};
\node (L)at (-1.4,-1){$1$};
\node (R) at(1.4,-1){$2$};
\node (LL)at(-2.4,-2){$1$};
\node (LR)at(-0.4,-2){$4$};
\node (RL)at (0.4,-2){$2$};
\node (RR)at(2.4,-2){$3$};
\draw(root)--(L)--(LL)(L)--(LR)(root)--(R)--(RL)(R)--(RR);
}\quad
\tikzpic{-0.5}{[scale=0.6]
\node(root)at(0,0){};
\node(L)at(-1,-1){$1$};
\node(LL)at(-2,-2){$1$};
\node(LR)at(0,-2){$3$};
\node(LLL)at(-3,-3){$1$};
\node(LLR)at(-1,-3){$4$};
\node(R)at(1,-1){$2$};
\draw(root)--(L)--(LL)--(LLL)(L)--(LR)(LL)--(LLR)(root)--(R);
}
\end{align*}
A permutation corresponding to the labeled binary tree in the middle of the second row 
is $1423$. 
An integer $m\in[2,4]$ appears exactly once in a right node in each labeled binary tree. 

We have five complete binary trees in $\mathcal{T}_{4}$.
The right labeled binary tree in the first row has the same complete binary tree as the 
middle labeled binary tree in the second row.
\end{example}

Let $\sigma$ be a permutation of $[n]$ with $\sigma_1=1$.
One can construct the circular presentation $I(\sigma)$ from 
a labeled binary tree $T(\sigma)$ as follows. 
Let $(\alpha_i,\beta_i)$, $1\le i\le n-1$, be the labels of 
the left and right child of some node.
In the circular presentation, we connect $\alpha_i$ and $\beta_i$
by a line for all $1\le i\le n-1$. 
The middle labeled binary tree in the second row in Example \ref{ex:LBTk1}
consists of the labels $(1,4)$, $(2,3)$ and $(1,2)$.
The circular presentation is given by $\{(1,2),(1,4),(2,3)\}$.
It is easy to see that the pairs $(\alpha_i,\beta_i)$, $1\le i\le n-1$, satisfy the 
conditions for the circular presentation.

As a summary, we have the following proposition.
\begin{prop}
Let $\sigma$ be a permutation as above.
We have a natural bijection between $I(\sigma)$ and $T(\sigma)$.
\end{prop}

Recall that the M\"obius function of the partition lattice $\Pi_{n}$ 
is equal to $(-1)^{n-1}(n-1)!$.
By Proposition \ref{prop:Mobiusdc}, the number of maximal decreasing chains 
is also equal to $(n-1)!$.
Since the total number of the labeled tree $T(\sigma)$ is $(n-1)!$,
we have a natural bijection between a labeled tree and a maximal decreasing 
chain.
Suppose that the labels of left and right children of a node $c$ 
are $\alpha$ and $\beta$, respectively.
Then, we say that the node $c$ has a label $(\alpha,\beta)$.
Given a labeled tree $T(\sigma)$, the set of labels $(\alpha,\beta)$ of 
the nodes can be arranged to be a decreasing chain.
Then, since the label of the parent node of $c$ is equal to that of 
its left child from the condition ($\clubsuit3$), it is a routine to check that the chain satisfies the property 
in Lemma \ref{lemma:betaperm}.
As a summary we have the following:
%%%%%%
\begin{prop}
A labeled binary tree $T(\sigma)$ is bijective to a maximal decreasing chain 
of the lattice of the partitions.
\end{prop}

Given a labeled binary tree $T(\sigma)$, the labels  
on internal nodes of $T(\sigma)$ satisfy the following properties:
\begin{enumerate}[(P1)]
\item A label of a node $c$ is equal to the minimum of the labels 
of leaves in $T_{c}$ where $T_{c}$ is a subtree of $T(\sigma)$ whose 
root is $c$.  
\end{enumerate}
Note that the property (P1) is compatible with the definition (\ref{eq:alphabeta}) of a label.

\begin{remark}
In a labeled binary tree, labels in right nodes take values in $\{2,3,\ldots,n\}$ and 
each integer appears exactly once. This fact comes from the conditions ($\diamondsuit1$) 
and ($\diamondsuit2$) in Section \ref{sec:cr}. On the contrary, the multiplicity of 
an integer in $\{1,2,\ldots,n-1\}$ on left nodes may be larger than one or zero.
\end{remark}

\begin{example}
The six maximal decreasing chains for $(n,k)=(4,1)$ corresponding to the 
labeled binary trees in Example \ref{ex:LBTk1} are 
\begin{align*}
&\genfrac{(}{)}{0pt}{}{3}{4}\genfrac{(}{)}{0pt}{}{2}{3}\genfrac{(}{)}{0pt}{}{1}{2}, \qquad
\genfrac{(}{)}{0pt}{}{2}{4}\genfrac{(}{)}{0pt}{}{2}{3}\genfrac{(}{)}{0pt}{}{1}{2}, \qquad
\genfrac{(}{)}{0pt}{}{2}{4}\genfrac{(}{)}{0pt}{}{1}{3}\genfrac{(}{)}{0pt}{}{1}{2}, \\
&\genfrac{(}{)}{0pt}{}{3}{4}\genfrac{(}{)}{0pt}{}{1}{3}\genfrac{(}{)}{0pt}{}{1}{2}, \qquad
\genfrac{(}{)}{0pt}{}{2}{3}\genfrac{(}{)}{0pt}{}{1}{4}\genfrac{(}{)}{0pt}{}{1}{2}, \qquad
\genfrac{(}{)}{0pt}{}{1}{4}\genfrac{(}{)}{0pt}{}{1}{3}\genfrac{(}{)}{0pt}{}{1}{2}.
\end{align*}
\end{example}

Fix $k\ge2$. 
We will show a correspondence between a labeled binary tree and a maximal decreasing 
chain of $\mathcal{L}_{n}^{(k)}$.
Let $T\in\mathcal{T}_{n}$ be a complete binary tree with $n$ leaves.
Suppose that an internal node $c$ of $T$ has two binary trees, a left tree $B_{L}(c)$ and 
a right subtree $B_{R}(c)$. We denote by $c_L$ and $c_{R}$ the root nodes of $B_{L}(c)$ and $B_{R}(c)$.
Suppose that a labeled tree $L(T)$ (resp. $L'(T')$) has $n$ labels $l_1<l_2<\ldots<l_{n}$ 
(resp. $l'_1<l'_2<\ldots<l'_{n}$).
We say that a labeled tree $L(T)$ is {\it isomorphic} to $L'(T')$ if they have the same complete binary 
tree $T$, i.e., $T=T'$ and we obtain $L'(T')$ from $L(T)$ by mapping the labels $l_i$ to $l'_i$. 
We say that a labeled tree $L(T)$ is {\it canonical} if the labels of $L(T)$ are in $[n]$.

A labeled binary tree $L(T)$ for $k\ge2$ satisfies the following conditions:
\begin{enumerate}[($\spadesuit1$)]
\item A node takes a value $l$ in $[n]$ with a subscript $k'\in[k]$, i.e., the label 
is of the form $l_{k'}$. 
Each integer $l\in[n]$ appears exactly once in $n$ leaves.
\item Suppose the labels of $c_{L}$ and $c_{R}$ are $l_{k'}$ and $m_{k''}$.
\begin{enumerate}
\item We have $l<m$.
\item The subscripts of these two nodes are the same, i.e., $k'=k''$.
\end{enumerate}
These mean that the integer label of $c_L$ is smaller than that of $c_{R}$ in the 
linear order (\ref{eq:linord}).
\item The subscripts are weakly increasing from a leaf to the root.
\item
We say that the node $c$ is a {\it descent} if the label $l_{k'}$ of $c$ is smaller than 
the label $l'_{k''}$ of the left child.
If $c$ is a descent, we have $k'>k''$.
\item Suppose a node $c$ is not the root.
\begin{enumerate}
\item If $c$ is a right child of some node, the integer label of $c$ is equal to 
an integer label in $B_L(c)$ and $B_R(c)$, which is not a label of a right node of some node.
\item If $c$ is a left child of some node, the the integer label of $c$ is 
equal to one of integer labels in $B_L(c)$ and $B_R(c)$.
\end{enumerate} 
\item Suppose that $L_{i}$ be a labeled subtree of $L(T)$ whose labels have a subscript $i\in[k]$.
Then, $L_{i}$ is isomorphic to a canonical labeled tree.
\end{enumerate}

\begin{remark}
Some remarks are in order.
\begin{enumerate}
\item
From the conditions ($\spadesuit1$) and ($\spadesuit2$), the left-most leaf of 
$L(T)$ has an integer label $1$, whose subscript is in $[k]$.
\item The condition ($\spadesuit1$) implies that we may have the same labels on internal nodes.
\item Let $l_{q}$ be a label of a node $c$ and $l'_{p}$ a label of the left child of $c$.
The conditions ($\spadesuit3$) and ($\spadesuit4$) imply that $p<q$ if $l>l'$.
Therefore, labels of the nodes are weakly decreasing from a leaf to the root.
\item The condition ($\spadesuit5$) implies that a label in $[2,n]$ appears exactly 
once as a label of the right child of some node.
\item The conditions ($\spadesuit4$) and ($\spadesuit5$) are a generalization of 
the property (P1) for $k\ge2$.
\end{enumerate}
\end{remark}

\begin{defn}
\label{defn:LBT}
We denote by $\mathcal{LT}^{(k)}_{n}$ the set of labeled binary trees $L(T)$ 
satisfying the conditions ($\spadesuit1$) to ($\spadesuit6$) such that  
it has $n$ leaves with at most $n-2$ descents and the label of the left child of the root 
is not $1_{k}$.
\end{defn}

As we will see below, the last condition in Definition \ref{defn:LBT} is necessary to 
consider a maximal decreasing chain from $\hat{0}$ to $\hat{1}$ in the poset $\mathcal{L}_{n}^{(k)}$.
In fact, this condition is compatible with the fact that we have a label $(1,n)_{k}$ from $\pi$ 
to $\hat{1}$ where $\rho(\pi)=n-1$.

\begin{example}
\label{ex:lbt}
Let $(n,k)=(4,2)$ and $L(T)$ the following labeled binary tree:
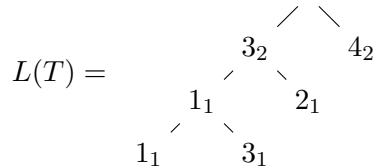
\begin{align*}
L(T)= 
\tikzpic{-0.5}{[scale=0.7]
\node (root) at (0,0){};
\node (L) at (-1,-1){$3_2$};
\node (LL) at (-2,-2){$1_1$};
\node (LLL) at (-3,-3){$1_1$};
\node (R) at (1,-1){$4_2$};
\node (LR) at (0,-2){$2_1$};
\node (LLR) at (-1,-3){$3_1$};
\draw (root)--(L)--(LR)(root)--(R)(L)--(LL)--(LLR)(LL)--(LLL); 
}
\end{align*}
We have one descent at the internal node with a label $3_2$.

Simiarly, the labeled binary tree
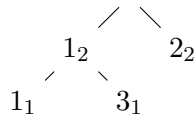
\begin{align*}
\tikzpic{-0.5}{[scale=0.7]
\node (root) at (0,0){};
\node (L) at (-1,-1){$1_2$};
\node (LL) at (-2,-2){$1_1$};
\node (R) at (1,-1){$2_2$};
\node (LR) at (0,-2){$3_1$};
\draw (root)--(L)--(LR)(root)--(R)(L)--(LL);
}
\end{align*}
has a descent at the node $c$ with the label $1_2$. Note that the label $1_2$ at $c$ 
has the same integer label as the label $1_1$ of the left child of $c$, but a different subscript.
\end{example}

\begin{theorem}
\label{thrm:mdclbt}
A maximal decreasing chain in $\mathcal{L}^{(k)}_{n}$ is bijective 
to a labeled binary tree in $\mathcal{LT}_{n}^{(k)}$.
\end{theorem}
%%%%%%%%%%%%
\begin{remark}
\label{remark:LBT}
The condition that the label of the left child of the root is not $1_{k}$ in 
Definition \ref{defn:LBT} implies that all the labels constructed from 
a labeled binary tree is larger than $(1,n)_{k}$ in the linear 
order (\ref{eq:deflex}). 
Therefore, if we construct a decreasing chain 
from a labeled tree, it is a decreasing chain in $\mathcal{P}_{n}^{(k)}$.
Since $\mathcal{L}_{n}^{(k)}$ is obtained from $\mathcal{P}_{n}^{(k)}$ by 
adding $\hat{1}$, a decreasing chain in $\mathcal{P}_{n}^{(k)}$ constructed 
from a labeled tree can be naturally extended to a maximal decreasing chain in $\mathcal{L}_{n}^{(k)}$.
On the contrary, if the label of the left child of the root is $1_{k}$, this sequence  of labels
can not be a maximal decreasing chain 
in the lattice $\mathcal{L}_{n}^{(k)}$.
\end{remark}
%%%%%%%%%%%%
\begin{proof}[Proof of Theorem \ref{thrm:mdclbt}]
Pick an internal node $c$ with two leaves.
Suppose that $\alpha_{l}$ (resp. $\beta_{l}$) be the label of the left 
(resp. right) child of $c$.
From the condition ($\spadesuit2$), the two labels have the same subscript.
Then, we assign the label $(\alpha,\beta)_{l}$ to the node $c$.
Note that from the condition ($\spadesuit2$), we always have $\alpha<\beta$.
Then, we delete the two leaves from the tree and obtain a labeled tree 
with $n-1$ leaves.
By continuing this process until we delete all the leaves,
we obtain a sequence of labels. 
This sequence of labels corresponds to the order of deleted internal nodes.

We will show that one can construct a maximal decreasing chain from this sequence of labels.
When the node $c$ is a descent, by the conditions ($\spadesuit4$) and ($\spadesuit5$), 
the label of $c$ is strictly smaller than that 
of the left child $c_{L}$.
In the deletion of nodes, we delete the node $c_{L}$ first and then $c$.
Thus, we have a decreasing sequence for these two nodes.
Suppose we have several choices to pick an internal node $c$ with two leaves.
Let $c_1,c_2,\ldots, c_{p}$ be such internal nodes, and 
$A_{i}:=(\alpha_{i},\beta_{i})_{l}$ the label corresponding to the node $c_{i}$.
From the conditions ($\spadesuit1$), ($\spadesuit2$) and ($\spadesuit3$), it is easy to see that 
we have a unique order of $A_{i}$ such that 
\begin{align*}
A_{i_1}>A_{i_2}>\ldots>A_{i_p},
\end{align*}
where $(i_1,i_2,\ldots,i_{p})$ is a permutation in $[p]$.
This implies that we have a decreasing chain if we delete the internal nodes 
in the order of $c_{i_1}, c_{i_2},\ldots,c_{i_p}$. 
After deleting $p$ internal nodes, we continue to delete nodes in a similar 
way until we delete all nodes.
We obtain a decreasing chain consisting of $n$ labels. 
As already mentioned as above in Remark \ref{remark:LBT},
to construct a maximal chain of the lattice $\mathcal{L}_{n}^{(k)}$, we 
have to append a label $(1,n)_{k}$ to the chain.
Recall that we have $\pi\xrightarrow{(1,n)_{k}}\hat{1}$ where 
$\rho(\pi)=n-1$.
By definition of the set $\mathcal{LT}^{(k)}_{n}$ (see Definition \ref{defn:LBT}), 
the label of the left child of the root is not $1_{k}$. 
This implies that all labels obtained from the labeled binary tree are larger 
than $(1,n)_{k}$ in the linear order (\ref{eq:linord}).
Combining these observations together, we have a unique decreasing chain 
corresponding to a labeled tree in $\mathcal{LT}^{(k)}_{n}$.

Conversely, suppose that a maximal decreasing chain is given.
Then, we first erase the label $(1,n)_{k}$ from the chain.
Let $l_1,\ldots,l_{n-1}$ be the labels of the decreasing chain.
We divide the labels into $k$ groups according to the subscript $k'\in[k]$.
Recall the linear order of labels is Eq. (\ref{eq:deflex}) and the chain 
is decreasing.
Let $L_{k'}$ be the chain of labels with the subscript $k'$.
The decreasing chain is expressed as $L_1L_2\ldots L_{k}$. 
Note that each $L_{k'}$ may contain several blocks as a weighted partition, and 
the labeled binary tree of an each block is isomorphic to a canonical tree.
Thus $L_{k'}$ may be a labeled forest.
To obtain a labeled tree, we construct a labeled tree from the labeled trees 
for blocks.
We put the labeled tree $L_{k'}$ on the top of $L_{k'-1}$ in such a way that 
a leaf of $L_{k'}$ is the root of $L_{k'-1}$ satisfying the condition ($\spadesuit5$).
By the condition ($\spadesuit1$), we have exactly one way to patch the two trees 
with subscripts $k'$ and $k'-1$.
Then, by construction, the newly obtained labeled tree satisfies 
all the conditions from ($\spadesuit1$) to ($\spadesuit6$).
Thus, we obtain a labeled tree from a maximal decreasing chain, which completes the proof.
\end{proof}

A labeled tree in Example \ref{ex:lbt} corresponds to 
the following maximal decreasing chain:
\begin{align*}
\hat{0}:\genfrac{(}{)}{0pt}{}{1}{3}_{1}\genfrac{(}{)}{0pt}{}{1}{2}_{1}
\genfrac{(}{)}{0pt}{}{3}{4}_{2}\genfrac{(}{)}{0pt}{}{1}{4}_{2}.
\end{align*}
We append the label $(1,4)_{2}$ to the deceasing chain obtained from 
the labeled tree.

From Proposition \ref{prop:Mobiusdc} and Theorem \ref{thrm:mdclbt}, 
we have the following:
\begin{cor}
The number of labeled binary trees is given by
\begin{align*}
\left|\mathcal{LT}_{n}^{(k)}\right|=(-1)^{n}\mu(\mathcal{L}_{n}^{(k)}).
\end{align*}
\end{cor}

\subsection{The Whitney numbers of the first kind and the characteristic polynomial}
The {\it Whitney numbers} $w_{r}^{(n,k)}$ of the first kind 
are defined by 
\begin{align*}
w_{r}^{(n,k)}:=\sum_{\pi\in\mathcal{P}_{n}^{(k)}(r)}\mu(\hat{0},\pi),
\end{align*}
where $\mathcal{P}_{n}^{(k)}(r)$ is the set of weighted partitions with $r$ blocks 
defined in Definition \ref{defn:Pnk}.
Note that $w_{n}^{(n,k)}=1$ since the least element ($\rho(\hat{0})=0$) is unique and 
it has $n$ blocks.
The {\it characteristic polynomial} $P_{n}^{(k)}(x)$ of $\mathcal{P}_{n}^{(k)}$ 
is defined as an ordinary generating function by 
\begin{align}
\label{eq:Pnkxw}
P_{n}^{(k)}(x):=\sum_{r=1}^{n}w_{r}^{(n,k)}x^{r}.
\end{align}
The specialization $x=1$ in Eq. (\ref{eq:Pnkxw}) gives 
\begin{align}
\label{eq:Pnkx1}
P_{n}^{(k)}(x=1)=-\mu(\mathcal{L}_{n}^{(k)}),
\end{align}
by the definition of the M\"obius function.

\begin{prop}
\label{prop:charpol}
The characteristic polynomial $P_{n}^{(k)}(x)$ of $\mathcal{P}_{n}^{(k)}$ 
is given by 
\begin{align}
\label{eq:Pnkprod}
P_{n}^{(k)}(x)=\prod_{j=0}^{n-1}(x-kj).
\end{align} 
\end{prop}
%%%%%%%%%%%%%
Before proceeding to the proof of Proposition \ref{prop:charpol}
we introduce the following lemma.

\begin{lemma}
\label{lemma:muprod}
Let $\pi\in\mathcal{P}_{n}^{(k)}$ be a weighted partition with $r$ blocks. 
We denote by $\pi_i$, $1\le i\le r$, the weighted partition corresponding to 
a block $B_{i}$.
Then, we have 
\begin{align}
\label{eq:pmu}
\mu(\hat{0},\pi)=\prod_{i=1}^{r}\mu(\hat{0},\pi_{i}).
\end{align}
\end{lemma}
%%%%%%%%%%%
\begin{proof}
From Proposition \ref{prop:Mobiusdc}, $\mu(\hat{0},\pi)$ is equal to 
the number of maximal decreasing chains from $\hat{0}$ to $\pi$.
Since $\pi$ consists of $r$ weighted partitions $\pi_{i}$, $1\le i\le r$, 
$\mu(\hat{0},\pi_i)$ is equal to the number of maximal decreasing 
chains from $\hat{0}$ to $\pi_i$.
Given maximal decreasing chains for $\pi_{i}$, $1\le i\le r$, we have a unique 
maximal decreasing chains from $\hat{0}$ to $\pi$.
Thus, it is straightforward to see that Eq. (\ref{eq:pmu}) holds.
\end{proof}

\begin{proof}[Proof of Proposition \ref{prop:charpol}]
When $n=2$, the weighted partitions are $\hat{0}=1/2$ and 
$(12)^{k'}$ with $1\le k'\le k$.
Then, it is easy to see that, by a simple calculation, we have 
\begin{align*}
P_{2}^{(k)}(x)&=x^2-kx, \\
&=x(x-k),
\end{align*}
which implies Eq. (\ref{eq:Pnkprod}) holds for $n=2$.

We prove the proposition by induction on $n$. 
Suppose Eq. (\ref{eq:Pnkprod}) holds for up to $n-1$.
We expand Eq. (\ref{eq:Pnkprod}) with respect to $x$ in the case of $n-1$. 
By definition of characteristic polynomials and of the Stirling numbers 
of the first kind, we have 
\begin{align}
\label{eq:wnkinStirS1}
w^{(n-1,k)}_{r}=k^{n-1-r}s(n-1,r),
\end{align}  
where $1\le r\le n-1$ and $s(n,r)$ is the Stirling number of the first kind defined in Eq.  (\ref{eq:defsn}).
Then, from Lemma \ref{lemma:muprod}, 
we have $\mu(\hat{0},\pi)=\mu(\hat{0},\pi_{1})\mu(\hat{0},\pi\setminus \pi_{1})$ 
where $\pi_1$ is a block of a weighted partition $\pi$ containing the point $1$ in the circular presentation 
and $\pi\setminus\pi_1$ is a weighted partition obtained from $\pi$ by deleting 
the block $\pi_1$.
If $\pi$ consists of $r$ blocks, $\pi_1$ consists of a single block and $\pi\setminus\pi_1$ consists of 
$r-1$ blocks.
If $\pi\setminus\pi_1$ consists of $i$ points, $\pi_1$ consists of $n-i$ points.
Since $i$ points in $\pi\setminus\pi_1$ do not belong to the same block as $1$, 
we have $\genfrac{(}{)}{0pt}{}{n-1}{i}$ ways to choose $i$ points.
The M\"obius functions for $\pi_1$ and $\pi\setminus\pi_1$ are given 
by the Whitney numbers: $\mu(\hat{0},\pi_1)=w^{(n-i,k)}_1$
and $\mu(\hat{0},\pi\setminus\pi_1)=w^{(i,k)}_{r-1}$.
Therefore, we have
\begin{align*}
w^{(n,k)}_{r}=\sum_{i=r-1}^{n-1}w^{(n-i,k)}_{1}\genfrac{(}{)}{0pt}{}{n-1}{i}w^{(i,k)}_{r-1},
\end{align*}
where $2\le r\le n$.
By substituting Eq. (\ref{eq:wnkinStirS1}) into the above expression and using 
the recurrence relation for the Stirling numbers of the first kind, 
we have $w^{(n,k)}_{r}=k^{n-r}s(n,r)$ for $2\le r\le n$.
From these observations, we have 
\begin{align*}
P_{n}^{(k)}(x)=\prod_{j=0}^{n-1}(x-kj) + C_{0}x, 
\end{align*}
where $C_{0}$ is a constant.
By comparing Eq. (\ref{eq:Pnkx1}) with Theorem \ref{thrm:Gamma}, 
we obtain $C_{0}=0$.
Thus, Eq. (\ref{eq:Pnkprod}) follows.
\end{proof}

\bibliographystyle{amsplainhyper} 
\bibliography{biblio}

\end{document}